\documentclass[twoside,a4paper,reqno,intlimits]{amsart}%
\usepackage[active]{srcltx}
\usepackage{color,esint}
\usepackage{amsthm}
\usepackage{amssymb}
\usepackage{amsfonts}

\usepackage{enumitem}
\setenumerate{label={\rm (\alph{*})}}

\usepackage[rose,frak,operator]{paper_diening}

\usepackage{graphicx}
\usepackage[%
   colorlinks=true,
   urlcolor=rltblue,       % \href{...}{...} external (URL)
   filecolor=rltgreen,     % \href{...} local file
   citecolor=rltgreen,     %~\cite{...} 
   linkcolor=rltred,       %~\ref{...} and \pageref{...}
   bookmarks=true,
   unicode,
   plainpages=false,
   ]{hyperref}
\usepackage{xcolor}
\definecolor{rltred}{rgb}{0.75,0,0}
\definecolor{rltgreen}{rgb}{0,0.5,0}
\definecolor{rltblue}{rgb}{0,0,0.75}

\newtheorem{theorem}[equation]{Theorem}
\newtheorem{lemma}[equation]{Lemma}
\newtheorem{proposition}[equation]{Proposition}

\newtheorem{assumption}[equation]{Assumption}
\newtheorem{remark}[equation]{Remark}

\numberwithin{equation}{section}

\newcommand{\meantmp}[2]{#1\langle{#2}#1\rangle}
\newcommand{\mean}[1]{\meantmp{}{#1}}

\newcommand{\param}{\lambda}

\newcommand{\para}{{\delta}}

 \newcommand{\td}{\partial_{\tau}}

\newcommand{\bue}{\bu_{\epsilon}}
\newcommand{\buen}{\bu_{\epsilon_n}}
\newcommand{\letter}{\kappa}

\newcommand{\dx}{{\rm d\bx}}
\newcommand{\bfuMh}{\mbox{\bf u}_h^m}
\newcommand{\bfvM}{\mbox{\bf v}_h}

\newcommand{\Int}{\dashint\limits_{\!\!I_m}}%\hspace*{-1.5mm}}

\newcommand{\ksumo}{{\kappa\sum_{m=1}^M}}

\newcommand{\trap}[1]{{#1}_{\tau}}
\newcommand{\difp}[1]{d^+{#1}}
\newcommand{\difn}[1]{d^-{#1}}
\newcommand{\intO}{\int\limits_{\Omega}}
\newcommand{\otimess}{\overset{s}{\otimes}}
\newcommand{\T}{\mathbf{S}}
\DeclareMathOperator{\spt}{spt}
\newcommand{\dist}{\operatorname{dist}}
\newcommand{\tran}[1]{{#1}_{-\tau}}

\begin{document}
\title[Space-time nonlinear parabolic systems]
{Space-time discretization for nonlinear parabolic systems with $p$-structure}
\author{Luigi C.\ Berselli}
\address{Dipartimento di Matematica, Universit{\`a} di Pisa, Via F.~Buonarroti 1/c,
  I-56127 Pisa, ITALY.}  \email{luigi.carlo.berselli@unipi.it}
%
%\and Lars Diening\footnotemark[2] 
%
\author{Michael R\r u\v zi\v cka{}} 
\address{Institute of Applied Mathematics, Albert-Ludwigs-University Freiburg,
  Ernst-Zermelo-Str.~1, D-79104 Freiburg, GERMANY.}  \email{ rose@mathematik.uni-freiburg.de}
\begin{abstract}
  In this paper we consider nonlinear parabolic systems with elliptic part which can be
  also degenerate. We prove optimal error estimates for smooth enough solutions. The
  main novelty, with respect to previous results, is that we obtain the estimates directly
  without introducing intermediate semi-discrete problems. In addition, we prove the
  existence of solutions of the continuous problem with the requested regularity, if the
  data of the problem are smooth enough.
\end{abstract}
\date{\small \today}
\maketitle
\section{Introduction}
In this paper we study the (full) space-time discretization of a parabolic problem with
Dirichlet boundary conditions. Our method differs from most previous investigations in
as much as we use no intermediate problems to prove an optimal error estimate. This result
is achieved under certain natural regularity assumptions of the solution of the continuous
problem. Moreover, we also prove this required regularity for the solution of the singular
problem for large data, in the case of Dirichlet boundary conditions. We restrict
ourselves to the three-dimensional setting, however, all results carry over to the general
setting in $d$-dimensions.
%\section{Setting of the problem}

More precisely, we consider for a sufficiently smooth bounded domain
$\Omega \subset \setR^3$ and a finite time interval $I:=(0,T)$, for
some given $T>0$, the parabolic system
\begin{equation}
  \label{eq:pfluid}\tag{$\text{parabolic}_p$}
  \begin{aligned}
    \frac {\partial\bfu}{\partial t} -\divo \bfS (\bfD\bfu) &= \bff\qquad&&\text{in }
    I\times \Omega,
    \\
    \bu &= \bfzero &&\text{on } I\times\partial \Omega\,,
    \\
    \bu(0)&=\bu_0&&\text{in }\Omega\,,
  \end{aligned}
\end{equation}
where the elliptic operator $\bS$ has $(p,\delta)$-structure and depends only on the
symmetric part of the gradient $\bD\bu$ of the vector-valued unknown $\bu:\Omega\to\setR^3$. Of
course, the whole theory also works with some simplifications if $\bS$ depends on the full
gradient $\nabla \bu$ and in an $d$-dimensional setting with $d\geq2$. The variational formulation
of~\eqref{eq:pfluid} is (for smooth enough solutions) the following
\begin{align}    
  \label{eq:cont-var}
  \begin{aligned}
    \Bighskp{\frac {\partial\bfu}{\partial t}(t)
    }{\bfv}+\hskp{\bS(\bD\bfu(t))}{\bD\bfv}&=\hskp{\bff(t)}{\bfv}
    && \forall\,\bfv\in V, \text{ a.e. } t \in I,
    \\
    \hskp{\bu(0)}{\bv}&=\hskp{\bu_0}{\bv}&& \forall\,\bfv\in V\,,
  \end{aligned}
\end{align}
where we will set, for reasons explained later, $V=(W^{1,p}_{0}(\Omega)\cap
L^{2}(\Omega))^{3}$.  We perform an error analysis for the fully implicit space-time
discretization
%\marginpar{Initial value}
\begin{align}  
  \label{eq:fdiscr-var}
  \begin{aligned}
    (d_t \bfuMh,\bfvM)+\hskp{\bS(\bD\bfuMh)}{\bfv_h} &=(\bff(t_m),\bfvM) &&\forall\, \bfvM
    \in V_h ,\,\  m=1,\ldots,M\,,
    \\
    \hskp{\bu_h^0}{\bv_h}&=\hskp{\bu_0}{\bv_h}&& \forall\,\bfv_h\in V_h\,,
  \end{aligned}
\end{align}
%\comment{See later for choice $\bu^0_h$}
where $d_t \bu^m:=\kappa^{-1} ({\bu^m-\bu^{m-1}})$ is the backward difference quotient
with $\kappa :=\frac TM$, $M\in \setN$ given, $t_m:=m\,\kappa$, and where $V_h\subset V$
is an appropriate finite element space with mesh size $h>0$. Precise definitions will be
given below.
\section{Notation and preliminaries}
\label{sec:preliminaries}
In this section we introduce the notation we will use. Moreover, we recall some
technical results which will be needed in the proof of the main convergence result.
\subsection{Function spaces}
We use $c, C$ to denote generic constants, which may change from line
to line, but are not
depending on the crucial quantities. Moreover we write $f\sim g$ if and only if there
exists constants $c,C>0$ such that $c\, f \le g\le C\, f$.

We will use the customary Lebesgue spaces $(L^p(\Omega), \norm{\,.\,}_p)$ and Sobolev
spaces $(W^{k,p}(\Omega), \norm{\,.\,}_{k,p})$, $k \in \setN$. We do not distinguish
between scalar, vector-valued or tensor-valued function spaces in the notation if there is
no danger of confusion. However, we denote scalar functions by roman letters,
vector-valued functions by small boldfaced letters and tensor-valued functions by capital
boldfaced letters.  If the norms are considered on a set $M$ different from $\Omega$, this
is indicated in the respective norms as $\norm{\,.\,}_{p,M}, \norm{\,.\,}_{k,p,M}$. We
equip $W^{1,p}_0(\Omega)$ (based on the \Poincare{} Lemma) with the gradient norm
$\norm{\nabla \,.\,}_p$.  We denote by $\abs{M}$ the $3$-dimensional Lebesgue measure of a
measurable set $M$. The mean value of a locally integrable function $f$ over a measurable
set $M \subset \Omega$ is denoted by $\mean{f}_M:= \dashint_M f \, dx =\frac 1
{|M|}\int\limits_M f \, dx$. Moreover, we use the notation $({f},{g}):=\int\limits_\Omega
f g\, dx$, whenever the right-hand side is well defined.
\subsection{Basic properties of the elliptic operator} 
\label{sec:stress_tensor}
For a tensor $\bfP \in \setR^{3 \times 3} $ we denote its symmetric part by $\bP^\sym:=
\frac 12 (\bP +\bP^\top) \in \setR_\sym ^{3 \times 3}:= \set {\bfA \in \setR^{3 \times 3}
  \,|\, \bP =\bP^\top}$. The scalar product between two tensors $\bP, \bQ$ is denoted by
$\bP\cdot \bQ$, and we use the notation $\abs{\bP}^2=\bP \cdot \bP $. We assume that the
extra stress tensor $\bS$ has $(p,\para)$-structure, which will be defined now.  A
detailed discussion and full proofs of the following results can be found
in~\cite{die-ett,dr-nafsa}.
\begin{assumption}
  \label{ass:1}
  We assume that ${\bS\colon \setR^{3 \times 3} \to \setR^{3 \times 3}_\sym }$ belongs to
  $C^0(\setR^{3 \times 3},\setR^{3 \times 3}_\sym )\cap C^1(\setR^{3 \times 3}\setminus
  \{\bfzero\}, \setR^{3 \times 3}_\sym ) $, satisfies ${\bS(\bP) = \bS\big (\bP^\sym \big
    )}$, and $\bS(\mathbf 0)=\mathbf 0$. Moreover, we assume that $\bS$ has {\rm
    $(p,\para)$-structure}, i.e., there exist $p \in (1, \infty)$, $\para\in [0,\infty)$,
  and constants $C_0, C_1 >0$ such that
   \begin{subequations}
     \label{eq:ass_S}
     \begin{align}
       \sum\nolimits_{i,j,k,l=1}^3 \partial_{kl} S_{ij} (\bP) Q_{ij}Q_{kl} &\ge C_0 \big
       (\para +|\bP^\sym|\big )^{{p-2}} |\bQ^\sym |^2,\label{1.4b}
       \\
       \big |\partial_{kl} S_{ij}({\bP})\big | &\le C_1 \big (\para +|\bP^\sym|\big
       )^{{p-2}},\label{1.5b}
     \end{align}
   \end{subequations}
   are satisfied for all $\bP,\bQ \in \setR^{3\times 3} $ with $\bA^\sym \neq \bfzero$ and
   all $i,j,k,l=1,\ldots, 3$.  The constants $C_0$, $C_1$, and $p$ are called the {\em
     characteristics} of $\bfS$.
\end{assumption}
\begin{remark}
  {\rm 
    We would like to emphasize that, if not otherwise stated, the constants in the
    paper depend only on the characteristics of $\bfS$ but are independent of
    $\delta\geq0$.  
}
\end{remark}
Another important tool are {\rm shifted N-functions} $\set{\phi_a}_{a \ge 0}$,
cf.~\cite{die-ett,die-kreu,dr-nafsa}. Defining for $t\geq0 $ a special N-function $\phi$
by
\begin{align} 
  \label{eq:5} 
  \varphi(t):= \int\limits _0^t \varphi'(s)\, ds\qquad\text{with}\quad
  \varphi'(t) := (\delta +t)^{p-2} t\,,
\end{align}
% and $\varphi(t):= \int\limits _0^t \varphi'(s)\, ds$, 
we can replace $C_i \big (\para +|\bP^\sym|\big )^{{p-2}}$ in the right-hand side
of~\eqref{eq:ass_S} by $\widetilde C_i\,\varphi'' \big (|\bP^\sym|\big )$, $i=0,1$. Next,
the shifted functions are defined for $t\geq0$ by
\begin{align}
  \label{eq:phi_shifted}
  \varphi_a(t):= \int\limits _0^t \varphi_a'(s)\, ds\qquad\text{with }\quad
  \phi'_a(t):=\phi'(a+t)\frac {t}{a+t}.
\end{align}
Note that $\phi_a(t) \sim (\delta+a+t)^{p-2} t^2$ and
also $(\phi_a)^*(t) \sim ((\delta+a)^{p-1} + t)^{p'-2} t^2$, where the $*$-superscript
denotes the complementary function.  We will use also the Young inequality: for all
$\varepsilon >0$ there exists $c_\epsilon>0 $, such that for all $s,t,a\geq 0$ it holds
 \begin{align}
   \label{ineq:young}
   \begin{split}
     ts&\leq \epsilon \, \phi_a(t)+ c_\epsilon \,(\phi_a)^*(s)\,,
     \\
     t\, \phi_a'(s) + \phi_a'(t)\, s &\le \epsilon \, \phi_a(t)+ c_\epsilon \,\phi_a(s).
   \end{split}
 \end{align}

 Closely related to the extra stress tensor $\bS$ with $(p,\delta)$-structure is the
 function $\bF\colon\setR^{3 \times 3} \to \setR^{3 \times 3}_\sym$ defined through
\begin{align}
  \label{eq:def_F}
  \bF(\bP):= \big (\para+\abs{\bP^\sym} \big )^{\frac {p-2}{2}}{\bP^\sym } \,.
\end{align}
In the following lemma we recall several useful results, which will be frequently used in
the paper. The proofs of these results and more details can be found
in~\cite{die-ett,dr-nafsa,die-kreu,bdr-phi-stokes}.
\begin{proposition}%[See Lemma~2.3 in~\cite{DieE08}]
  \label{lem:hammer}
  Let $\bfS$ satisfy Assumption~\ref{ass:1}, let $\phi$ be defined
  in~\eqref{eq:5}, and let $\bfF$ be defined in~\eqref{eq:def_F}.
  \begin{itemize}
  \item [\rm (i)] For all $\bfP, \bfQ \in \setR^{3 \times 3}$ 
    \begin{align*}
%      \label{eq:hammera}
        \big({\bfS}(\bfP) - {\bfS}(\bfQ)\big) \cdot \big(\bfP-\bfQ
        \big) &\sim \bigabs{ \bfF(\bfP) - \bfF(\bfQ)}^2,
        \\
%        \label{eq:hammerb}
        &\sim \phi_{\abs{\bfP^\sym}}(\abs{\bfP^\sym - \bfQ^\sym}),
        \\
%        \label{eq:hammerc}
        &\sim \phi''\big( \abs{\bfP^\sym} + \abs{\bfQ^\sym}
        \big)\abs{\bfP^\sym - \bfQ^\sym}^2,
      \\
%      \label{eq:hammerd}
      \bfS(\bfQ) \cdot \bfQ &\sim \abs{\bfF(\bfQ)}^2 \sim
        \phi(\abs{\bfQ^\sym}),
      \\
%        \label{eq:hammere}  
      \abs{\bfS(\bfP) - \bfS(\bfQ)} &\sim
        \phi'_{\abs{\bfP^\sym}}\big(\abs{\bfP^\sym - \bfQ^\sym}\big).
    \end{align*}
  The constants depend only on the characteristics of $\bfS$.
\item [\rm (ii)] For all $\epsilon>0$, there exist a constant $c_\epsilon>0$ (depending
  only on $\epsilon>0$ and on the characteristics of $\bfS$) such that for all $\bu,
  \bv,\bw \in W^{1,p}(\Omega)$ %we have
  \begin{align*}
    &\big ( {\bfS(\bD\bfu) - \bfS(\bD\bfv)},{\bD\bfw - \bD
      \bfv}\big )
      \leq \epsilon\, \norm{\bfF(\bD\bfu) - \bfF(\bD\bfv)}_2^2
      +c_\epsilon\,  \norm{\bfF(\bD\bfw) - \bfF(\bD\bfv)}_2^2\,,
    \\
      &\big ( {\bfS(\bD\bfu) - \bfS(\bD\bfv)},{\bD\bfw - \bD
      \bfv}\big )
      \leq \epsilon\, \norm{\bfF(\bD\bfw) - \bfF(\bD\bfv)}_2^2
      +c_\epsilon\,  \norm{\bfF(\bD\bfu) - \bfF(\bD\bfv)}_2^2\,,
\end{align*}
  and for all $\bfP,\bfQ\in\setR^{3 \times 3}_\sym$, $t\geq 0$
  \begin{align*}
    \phi_{\abs{\bfQ}}(t)&\leq c_\vep\, \phi_{\abs{\bfP}}(t)
    +\vep\, \abs{\bfF(\bfQ) - \bfF(\bfP)}^2,
    \\
    (\phi_{\abs{\bfQ}})^*(t)&\leq c_\vep\, (\phi_{\abs{\bfP}})^*(t)
    +\vep\, \abs{\bfF(\bfQ) - \bfF(\bfP)}^2\,.
  \end{align*}
  \item [\rm (iii)]   Let $\Omega$ be a bounded domain. Then, for all
  $\bfH \in L^p(\Omega)$ 
  \begin{align*}
    \int\limits_\Omega \abs{\bfF(\bfH) - \mean{\bfF(\bfH)}_\Omega}^2\, dx \sim  \int\limits_\Omega
    \abs{\bfF(\bfH) - \bfF(\mean{\bfH}_\Omega)}^2\, dx\,,
  \end{align*}
  where the constants depend only on~$p$.
  \end{itemize}
\end{proposition} 

There hold the following important equivalences, first proved in~\cite{SS00}. See
also~\cite[Proposition~2.4]{br-plasticity}.
\begin{proposition}
  \label{prop:equivalence}
  Assume that $\bS $ has $(p,\delta)$-structure. For $i=1,2,3$ and for
  sufficiently smooth symmetric tensor fields $\bQ$ we denote\footnote{Note that there is no
summation convention over the repeated Latin lower-case index $i$ in
$\partial_i\bS(\bQ)\cdot\partial_i\bQ$.} 
\begin{equation}
  \label{eq:P}
  \mathbb{P}_i(\bQ):=\partial_i\bS(\bQ)\cdot\partial_i\bQ =
  \sum_{k,l,m,n=1}^3\partial_{kl}S_{m n}(\bQ) \,\partial_i Q_{kl}\,\partial_i Q_{m n}\,.
\end{equation}
%If $\bS$ has $(p,\delta)$-structure, then we have immediately $ \mathbb{P}_i(\bQ)\geq0$,
%for $i=1,2,3$, but also the following equivalences are valid,
Then we have for all smooth enough
symmetric tensor fields $\bQ$ and all $i=1,2,3$
  \begin{align}
    &\mathbb{P}_i(\bQ)\sim\phi''(|\bQ|)|\partial_i
    \bQ|^2\sim|\partial_i \bF(\bQ)|^2\,,\label{eq:p-F}
    % \\
    % &|\partial_i \bS(\bD\bv)|^2\sim \phi''(|\bD\bv|)  \mathbb{P}_i(\bv),
    \\
    &\mathbb{P}_i(\bQ)\sim\frac{|\partial_i\bS(\bQ)|^2}{\phi''(|\bQ|)}\,,\label{eq:p-S}
%      \sim |\partial_i \bF(\bv)|^2.
\end{align}
where the constants only depend on the characteristics of $\bS$.
\end{proposition}

\subsection{Discretizations}
For the time-discretization, given $T>0$ and $M\in \setN$, we define the time step size as
$\kappa:=T/M>0$, with the corresponding net $I^M:=\{t_m\}_{m=0}^M$, $t_m:=m\,\kappa$. We
use the notation $I_m:=(t_{m-1},t_m]$, with $m=1,\ldots,M$.  For a given sequence
$\{\bv^m\}_{m=0}^M$ we define the backward differences quotient as
\begin{equation*}
  d_t \bv^m:=\frac{\bv^m-\bv^{m-1}}{\kappa}.
\end{equation*}
% For a function $\bv\in L^1(I\times \Omega)$ we define the sequence
% $\{\obfvm\}_{m=0}^M$,  the retarded-time-averages of $\bv$, as
% \begin{equation}
%   \label{eq:time_smoothing} 
%   \overline{\bv}^{0}{:=}\bv(0)\quad\text{and}\quad 
%   \overline{\bv}^{m}{:=}\frac{1}{\kappa}\int\limits_{t_{m-1}}^{t_m}
%   \bv(\sigma,x)\,d\sigma=\Int \bv(\sigma,x)\,d\sigma,\qquad m\geq 1.
% \end{equation}
% To deal with time-discrete problems we shall use the discrete spaces
% $l^p(I^M;X)$ consisting of $X$-valued sequences $\{a_m\}_{m=0}^M$,
% endowed with the weighted norm
% \begin{equation*}
%   \|a_m\|_{l^p(I^M;X)}:=
%   \left\{\begin{aligned}
%       &\left(\ksum\|a_m\|_X^p\right)^{1/p}\quad &&\text{if }1\leq
%       p<\infty
%       \\
%       &\max_{0\leq m\leq M}\|a_m\|_X &&\text{if }p=\infty.
%     \end{aligned}
%   \right.
% \end{equation*}

%\comment{In the discrete problem the domain is a polyhedron, while to prove regularity
%we need   $C^{2,1}$. Should we say something?}

For the spatial discretization we assume that $\Omega
\subset \setR^3$ is a polyhedral domain with Lipschitz continuous boundary.  Let
$\mathcal{T}_h$ denote a family of shape-regular triangulations, consisting of
$3$-dimensional simplices $K$. We denote by $h_K$ the diameter of $K$ and by $\rho_K$ the
supremum of the diameters of inscribed balls. We assume that $\mathcal{T}_{h}$ is
non-degenerate, i.e., $\max _{K \in \mathcal{T}_{h}} \frac {h_K}{\rho_K}\le \gamma_0$.  The
global mesh-size $h$ is defined by $h:=\max _{K \in \mathcal{T}_h}h_K$.  Let $S_K$ denote the
neighborhood of~$K$, i.e., $S_K$ is the union of all simplices of~$\mathcal{T}_{h}$
touching~$K$.  By the assumptions we obtain that $|S_K|\sim |K|$ and that the number of
patches $S_K$ to which a simplex belongs are both bounded uniformly in $h$ and $K$.

%Let us define %$ X := \big(W^{1,p}(\Omega)\big)^3$ and
%$ V := \big(W^{1,p}_0(\Omega)\cap L^2(\Omega)\big )^3$ 

 We denote by ${\mathcal P}_k(\mathcal{T}_{h})$, with $k \in \setN_0:=\setN\cup\{0\}$,
the space of scalar or vector-valued functions, which are polynomials of degree
at most $k$ on each $K\in \mathcal{T}_{h}$.  Given a triangulation $\mathcal{T}_h$ of $\Omega$ with
the above properties and given $r_0 \le r_1 \in \setN_0$ we denote by $X_{h}$ the space
\begin{equation*}
  X_{h}:=\left\{\bv\in (C(\overline{\Omega}))^{3}\fdg \bv\in \mathcal{P}\right\},
\end{equation*}
with ${\mathcal
  P}_{r_0}(\mathcal{T}_h) \subset \mathcal{P} \subseteq {\mathcal P}_{r_1}(\mathcal{T}_h)$.
% an appropriate conforming finite element space $X_h$ defined on $\mathcal{T}_h$, i.e., $X_h$
%satisfies $X_h \subset W^{1,1}_0(\Omega)$. Moreover, we set $V_h :=
%X_h \cap V$. 
Note that there exists a constant $c=c(r_1, \gamma_0)$
such that for all $\bv_h \in X_h$, $K \in \mathcal{T}_h$, $j\in
\setN_0$, and all $x \in K$ holds
\begin{align}
  \label{eq:inverse}
  \abs{\nabla ^j \bv_h(x)}\le c\, \dashint_K \abs{\nabla ^j \bv_h(y)}\, dy\,.
\end{align}
%and $H_h := X_h \cap H$.
% , while $\skp{\cdot}{\cdot}_h$ denotes the inner product in $V_h$.
%{\color{red} $N=\text{dim} V_h$} ?
For the weak formulation of the continuous and discrete problems we
will use the following function spaces
\begin{equation*}
V:=(W^{1,p}_{0}(\Omega)\cap
L^{2}(\Omega))^{3}\qquad\text{and}\qquad V_{h}:=V\cap X_{h}.
\end{equation*}

We also need some numerical interpolation operators.  Rather than
working with a specific interpolation operator we make the following
assumptions:
\begin{assumption}
  \label{ass:proj}
   We assume that $r_0=1$ %$\mathcal{P}_1(\mathcal{T}_h) \subset V_h$
  and that there exists a linear projection operator
  $P_h\colon (W^{1,1}(\Omega))^{3} \to X_h$ which 
  \begin{enumerate}
  % \item preserves divergence in the $Y_h^*$-sense, i.e.,
  %   \begin{align}
  %     \label{eq:div_preserving}
  %     \skp{\divergence \bfw}{\eta_h} &= \skp{\divergence \Pidiv
  %       \bfw}{\eta_h} \qquad \forall\, \bfw \in X,\; \forall\, \eta_h \in
  %     Y_h\,;
  %   \end{align}
  \item is locally $W^{1,1}$-stable in the sense that
    \begin{align}
      \label{eq:Pidivcont}
      \dashint_K \abs{P_h\bfw}\,dx &\leq c \dashint_{S_K}\!
      \abs{\bfw}\,dx + c \dashint_{S_K}\! h_K \abs{\nabla \bfw}\,dx
      \qquad \forall\, \bfw \in (W^{1,1}(\Omega))^{3},\; \forall\, K \in \mathcal{T}_h;
    \end{align}
  \item preserves zero boundary values, i.e., $P_h\colon (W^{1,1}_{0}(\Omega))^{3}\to
    (W^{1,1}_{0}(\Omega ))^{3}\cap X_{h}$.
  \end{enumerate}
\end{assumption}

Note that, e.g., ~the Scott-Zhang operator (cf.~\cite{zhang-scott})
satisfies this assumption. The properties of interpolation operators
$P_h$ satisfying Assumption \ref{ass:proj} are discussed in detail
in~\cite[Sec.~4,5]{dr-interpol}, \cite[Sec.~3.2]{bdr-phi-stokes}. We
collect the for us relevant properties in the next proposition. 
\begin{proposition}\label{prop:Ph}
  Let $P_h$ satisfy Assumption~\ref{ass:proj}. 
  \begin{enumerate}
  \item [\rm (i)]   Let $\bF(\bD \bv) \in W^{1,2}(\Omega)$. Then there exists a constant
    $c=c(p,r_1,\gamma_0)$ such that 
  \begin{align*}%\label{eq:0cont}
    \norm{\bF(\bD\bv) -\bF(\bD P_h\bv)}_2
    &\le c\, h\, \norm{\nabla \bF(\bD \bv ) }_2. 
  \end{align*}
  \item [\rm (ii)] Let $q\in[1,2)$ and $\ell =1$ or $\ell =2$ be such
    that $W^{\ell,q}(\Omega) \vnor \vnor L^2(\Omega)$. Then, there exists a constant
    $c=c(q,\ell,r_1,\gamma_0)$ such that for all $\bv \in
    W^{\ell,q}(\Omega)$ holds
    \begin{align*}%      \label{eq:int-error}
      \|\bv-P_h\bv\|_2\le c\,  h^{\ell +3( \frac 12 -\frac 1q)}\,\|\nabla^\ell \bv\|_q\,.
    \end{align*}
  \item [\rm (iii)] Let $\bF(\bD \bv) \in
    W^{1,2}(\Omega)$ and  $\bF(\bD \bw) \in
    L^{2}(\Omega)$. Then, there exists a constant
    $c=c(p,r_1,\gamma_0)$ such that % for almost all $s,t \in I_m$, $m=1,\ldots, M$, we have
% \comment{In fact here we are not using any time regularity, this comes later on with
%   integration. Should be better say $\bF(\bD \bv(t))\in  W^{1,2}(\Omega)$ for a.e. $t\in
%   I$??

% This has to be done accordingly with statement of Lemma~\ref{lem:interpolation}.
%}
\begin{align*}%  \label{eq:interpolation}
  \begin{aligned}
    &\int\limits_\Omega \phi_{|\bD\bfv|}\big(\big|\bD P_h{\bv}
    -\bD P_h\bfw \big|\big)\,\dx
\leq c \,h^2\|\nabla\bF(\bD\bv)\|^2_2
    +c\,\|\bF(\bD\bv)-\bF(\bD\bw)\|^2_2 \,,
      \end{aligned}
    \end{align*}
where the constants depends only on $\gamma_0$ and $p$.
\end{enumerate}
\end{proposition}
\begin{proof}
  The first assertion is proved e.g.~in~\cite[Cor.~5.8]{dr-interpol}. The second assertion
  is a generalization of the well known approximation property if on both sides there
  would be the same exponent $q$. Assertion (ii) will be proved in a more general context
  in the Appendix. Also assertion (iii), which is of more technical character, will be
  proved in the Appendix.
\end{proof}
\subsection{Main results}
%
%Before giving the proof of the main result, we recall some regularity
%and convergence results we will need in the sequel.
Let us now formulate the main result, proving optimal convergence rates for the error
between the solution $\bu$ of the continuous problem~\eqref{eq:pfluid} and the discrete
solution $\{\bfuMh \}_{m=0}^{ M}$ of the space-time
discretization~\eqref{eq:fdiscr-var}. Observe that the existence and uniqueness of the
solution $\{\bfu^m_h \}_{m=0}^{ M}$ of the discrete problem~\eqref{eq:fdiscr-var} follows directly from
the assumptions on the operator. Moreover, testing
\eqref{eq:fdiscr-var} with $\{\bfu^m_h \}_{m=0}^{ M}$ yields the energy estimate
\begin{equation*}
  \max_{m=1,\dots,M}\|\bfu^m_h\|^2_2+\ksumo\|\bF(\bD\bfu^m_h)\|^2_2\leq C,
\end{equation*}
for some constant independent of $h,\kappa$.

%\comment{The Gronwall lemma requires smallness of $\kappa$ to absorb last term, statement updated.}
\begin{theorem}
\label{thm:theorem-parabolic}
Let the tensor field $\bS$ in~\eqref{eq:pfluid} have $(p,\delta)$-structure for some
$p\in(1,2]$, and $\delta\in[0,\infty)$ fixed but arbitrary, and let $\Omega\subset\setR^3$
be a bounded polyhedral domain with Lipschitz continuous boundary. Assume that $\ff \in %{L^{p'}(I;L^{p'}(\Omega))} \cap
{W^{1,2}(I;L^2(\Omega))}$, $\bu_0 \in W^{1,3p/2}_0(\Omega)$ and that the solution $\bu$
of~\eqref{eq:pfluid} satisfies~\eqref{eq:cont-var} and
\begin{equation}\label{eq:reg-ass} 
  {\bF(\bD \bu)}\in W^{1,2}(I \times\Omega).
\end{equation}
Let the space $V_h$ be
defined as above with $r_0=1$ and let $\{\bfuMh\}_{m=0}^{ M}$ be solutions of~\eqref{eq:fdiscr-var}. Then there exists
$\kappa_0\in (0,1]$ such that for given $h\in ( 0, 1)$, $\kappa \in ( 0, \kappa_0)$, satisfying
\begin{equation}
  \label{eq:compatbility}
  %\exists\,c>0:\qquad
  h^{4/p'}\leq \sigma_0\, \kappa,
\end{equation}
for some $\sigma_{0}>0$, we have the following error estimate
\begin{equation*}
  \max_{m=1,\ldots,M}
  \|\bfuMh-\bfu(t_{m})\|_2^2+\ksumo\|\bF(\bD\bfuMh)-\bF(\bD\bfu(t_{m}))\|_2^2\leq
  c\,(h^2+\kappa^2),
\end{equation*}
where the constant $c$ depends only on the characteristics of $\bS$,
$\norm{\bF(\bD \bu)} _{W^{1,2}(I \times\Omega)}$,
$\norm{\partial _t\ff}_{L^2(I;L^2(\Omega)}$, $\norm{\bu_0}_{1,2}$, $\gamma_0$, $r_1$,
$\delta$, and $\sigma_0$.
\end{theorem}
% \comment{ $\bu_0$ ?? Observe that if $    \hskp{\bu_h^0}{\bv_h}=\hskp{\bu_0}{\bv_h}$,
%   then in the error       
%   \begin{equation*}
%     \begin{aligned}
%       ({d_t}(\bfu^1_h-\bfu(t_{1})),\bfu^1_h-P_h\bfu(t_{1}))&=\frac{1}{\kappa}((\bfu^1_h-
%       \bfu(t_{1})),\bfu^1_h-P_h\bfu(t_{1}))-\frac{1}{\kappa}(\bfu^0_h-\bfu_0,\bfu^1_h-P_h\bfu(t_{1})) 
%       \\
%       &=\frac{1}{\kappa}((\bfu^1_h-\bfu(t_{1})),\bfu^1_h-P_h\bfu(t_{1})),\quad(\text{being
%       }\bfu^1_h-P_h\bfu(t_{1})\in V_h)
%     \end{aligned}
% \end{equation*}
% hence the summation does not contain the term $\|\bfu^0_h-\bfu_0\|_2^2$ and starts from $m=1$.

% \medskip

% I give alternate formulations of the theorem: same hypotheses but if
% in~\eqref{eq:fdiscr-var} $\bu^0_h=P_h\bu_0$ and $\bu_0\in W^{1,2}$, then statement is
% again
% \begin{equation*}
%   \max_{m=1,\ldots,M}
%   \|\bfuMh-\bfu(t_{m})\|_2^2+\ksumo\|\bF(\bD\bfuMh)-\bF(\bD\bfu(t_{m}))\|_2^2\leq
%   c\,(h^2+\kappa^2),
% \end{equation*}
% or 
% \begin{equation*}
%   \max_{m=0,\ldots,M}
%   \|\bfuMh-\bfu(t_{m})\|_2^2+\ksumo\|\bF(\bD\bfuMh)-\bF(\bD\bfu(t_{m}))\|_2^2\leq
%   c\,(h^2+\kappa^2),
% \end{equation*}
% If in addition $\bF(\bD\bu_0)\in W^{1,2}$, then with this choice of $\bu_0$ we can also
% write
% \begin{equation*}
%   \max_{m=0,\ldots,M}
%   \|\bfuMh-\bfu(t_{m})\|_2^2+\ksum\|\bF(\bD\bfuMh)-\bF(\bD\bfu(t_{m}))\|_2^2\leq
%   c\,(h^2+\kappa^2).
% \end{equation*}
% They are all correct I think. Which one do you prefer?}
%
% \comment{in this way we have the weakest assumptions on $\bu_0$. But
%   to have summation from 0 would be also nice.  L. Keep as it is} 

\begin{remark}
  An optimal error estimate for problem~\eqref{eq:pfluid} with a
  nonlinearity depending on the full gradient $\nabla \bu$ under
  slightly different assumptions has been proved in~\cite{der} for
  $p>\frac{2d}{d+2}$. The case of evolutionary $p$-Navier-Stokes
  equations, where the nonlinearity depends on the symmetric gradient
  $\bD\bu$, has been treated in~\cite{pr,dpr1,dpr2,bdr-3-2} in the
  case of space periodic boundary conditions. The evolutionary
  $p$-Stokes equations have been treated in~\cite{sarah-phd} in the
  case of Dirichlet boundary conditions. All these results treat
  intermediate semi-discrete problems, for which a certain regularity
  has to be proved, to obtain the desired optimal convergence rates. This in fact limits the
  results in~\cite{pr,dpr1,dpr2,bdr-3-2} to the case of space periodic
  boundary conditions. Here we avoid such problems by proving the
  error estimate directly without using intermediate semi-discrete
  problems. The approach can be extended to the treatment of
  $p$-Navier-Stokes equations, which will be done in a forthcoming
  paper.

  In \cite{tscherpel-phd}, \cite{sueli-tscherpel} the convergence of a
  fully implicite space-time discretization (without convergence rate
  but also with no assumptions of smoothness of the limiting problem)
  of the evolutionary $p$-Navier-Stokes equations in the case of
  Dirichlet boundary conditions is proved. The
  convergence of the same numerical scheme~\eqref{eq:fdiscr-var}
  towards a weak solution has been recently proved in~\cite{BR2019}
  even for general evolution equations with pseudo-monotone operators.

We wish also to mention the recent results in \cite{BM2020} concerning
the parabolic problem with a variable exponent.
\end{remark}
% \comment{I am not really happy with the rem. we should write it better. L. I am happy
%   enough with the remark, but you can change at your taste}
The regularity assumed in~\eqref{eq:reg-ass} is natural in the sense
that under certain circumstances the existence of such solutions can
be proved. 

\begin{theorem}%(cf.~\cite[Thm.~5.1]{BDR2010})
  \label{thm:MT}
  Let the tensor field $\bS$ in~\eqref{eq:pfluid} have $(p,\delta)$-structure for some
  $p\in(1,2]$, and $\delta\in[0, \infty)$ fixed but arbitrary, and let
  $\Omega\subset\setR^3$ be a bounded domain with $C^{2,1}$ boundary. Assume that
\begin{equation*} 
  {\bff}\in {L^{p'}(I;L^{p'}(\Omega))} \cap {W^{1,2}(I;L^2(\Omega))},
\end{equation*}
and
\begin{equation*}
  {\bfu_0} \in{W^{2,2}(\Omega)}\cap W^{1,2}_0(\Omega), \text{ with }\divo\bS(\bD\bu_0)\in
  L^2(\Omega). 
\end{equation*}
Then, the system~\eqref{eq:pfluid} has a unique regular
solution, i.e.,
%\marginpar{strong sol. not defined}
$\bfu \in L^p(I;W^{1,p}_0 (\Omega))$ fulfils    % , satisfying for
  % a.e.~$t \in I'$ and for all $\bvarphi \in W^{1,p}_{\divergence}
%(\Omega)$
%  \begin{align}    \label{eq:18}
%   \hspace*{-1mm} \int\limits_\Omega{\partial_t\bfu(t)}\cdot{\bvarphi} + {\bfS(\bfD
%      \bfu(t))}\cdot {\bfD \bvarphi}  + {(\bfu(t) \cdot
%      \nabla\bfu(t)})\cdot { \bvarphi}\, dx &=
%    \int\limits_\Omega{\bff(t)}\cdot{ \bvarphi}\, dx,    
%  \end{align}
%  and 
% such that 
  \begin{align}
    \label{ineq:locally_strong_for_7_over_5_reg1}
    \begin{split}
     \norm{ \bfu}_{W^{1,\infty}(I; L^2(\Omega))} +
      &\norm{\bF(\bD\bfu)}_{W^{1,2}(I\times\Omega)} 
%      \norm{\bF(\bD\bfu)}_{L^{2\frac{5 p-6}{2-p}}(I;W^{1,2}(\Omega))}
      \leq c_0\,,
    \end{split}
  \end{align}
  where\footnote{Note that the dependence of the constant $c_0$ on $\delta$ is
    such that  $c_0(\delta)\le c_0(\delta_0)$ for all $\delta\le \delta_0$.  }  $c_0$ depends only on the characteristics of
  $\bS$, $\delta$, $T$,  $\Omega$, $ \|\bu_0\|_{2,2}$, $\norm{\divo
    \bS(\bD\bu_0)}_2$, $ \norm{\ff}_{L^{p'}(I \times\Omega)}$, $ \norm{\ff}_{L^{2}(I\times \Omega)}$, $  \norm{\frac{\partial\ff}{\partial
      t}}_{L^{2}(I\times \Omega)}$, and
  satisfies   \eqref{eq:cont-var} with ${V=W^{1,p}_0(\Omega)\cap L^2(\Omega)}$.  
%  with $c_0$ independent of $\delta$. In particular this implies,
%  uniformly in $\delta$,
%  \begin{subequations}
%    \label{sob}
%    \begin{align}
%      &\bfu\in
%      \setL^{\frac{p(5p-6)}{2-p}}(I;\setW^{2,\frac{3p}{p+1}}(\Omega))
%      \cap  C(I;\setW^{1,r}(\Omega))\qquad %\text{for all }
%      1\leq  r<6(p-1),\label{boundedness}
%      \\
%      &\partial_t\bfu \in \setL^\infty(I;\setL^2(\Omega))\cap
%      \setL^{\frac{p(5p-6)}{(3p-2)(p-1)}}(I;W^{1,\frac{3p}{p+1}}(\Omega))\,.
%      %\cap L^2(I;L^{\frac{12}{8-3p}}(\Omega ))
%      \label{time_derivative}
%    \end{align}
%  \end{subequations}
%with constant $c_0,c_1,c_2$ independent of $\parameter$.
%  with norms of~$\bfu$ and $\partial_t\bfu$ bounded by constants
 % $c_1=c_1(\parameter_0,p,C_0,\ff, \bfu_0,T,\Omega,r)$ and
 % $c_2=c_2(\parameter_0,p,C_0,\ff, \bfu_0,T,\Omega)$, respectively.
\end{theorem}
\begin{remark}
  In the literature exist several regularity results which are related
  to Theorem~\ref{thm:MT}. In most cases the regularity is studied for
  a scalar equation and/or in the steady case with a nonlinearity
  depending on the full gradient. The main difficulty of our problem is
  the regularity near the boundary in normal direction. Results in
  the interior are rather standard, since they can be considered as a special
  sub-case of the problem with space periodic boundary conditions
  (cf.~\cite{bdr-7-5} and references therein).  Most results treating
  a nonlinearity depending on the symmetric gradient strongly rely on
  the non-degeneracy of the elliptic operator ($\delta>0$) and more
  regular data, see e.g.~\cite{pruess}. In
  addition, some results concern the case $p>2$ (cf.~\cite{mnr2}, \cite{hugo-petr-rose}), while we are
  here considering the case $p\in(1,2)$ which has some very special
  features already in the steady case.

  The results proved here are not covered in the classical literature.
  A crucial fact is that our problem does not contain a
  divergence-free constraints. This allows us to prove optimal
  regularity results up to the boundary (cf.~\cite{br-plasticity} for
  a treatment of the steady case). For recent results on a related
  parabolic system cf.~\cite{CGM2019}, \cite{CM2013} and references therein.
\end{remark}

The regularity $\bF(\bD\bu)\in W^{1,2}(I\times \Omega)$ can be formulated in
terms of Bochner--Sobolev spaces. From~\cite[Thm.~33]{dr-7-5} and
standard embedding results it follows
% we obtain in particular
\begin{equation*}\label{eq:embed}
   \bF(\bD\bu) \in  L^\infty({I};L^3(\Omega)).
\end{equation*}
Since {$|\bD \bfu|^{p/2}+\delta^{\frac p2}
\sim|\bF(\bD\bu)|+\delta^{\frac p2} $ we get, by H\"older's inequality,}
\begin{equation*}% \label{eq:I-interpolation-u}
 \bfu\in L^\infty(I;W^{1,3p/2}(\Omega)).
\end{equation*}
% Next, we estimate $\partial_t\nabla\bfu$ in Lebesgue spaces, by using the information
% in~\eqref{eq:I-interpolation-u}. By using~
In~\cite[Lemma~4.5]{bdr-7-5} it is shown that 
\begin{align*}%\label{eq:reg}
%  \begin{aligned}
    \|\nabla^2\bfu\|_{\frac{6p}{4+p}}^2
    &\leq c\, \|\nabla \bF(\bD\bu)\|_2^2(\delta+\|\nabla\bu\|_{3p/2})^{2-p},
    \\
    \Bignorm{\frac {\partial \nabla \bfu}{\partial t}}_{\frac{6p}{4+p}}^2
    &\leq c\, \|\partial _t \bF(\bD\bu)\|_2^2(\delta+\|\nabla\bu\|_{3p/2})^{2-p}\,.
%  \end{aligned}
\end{align*}
% hence for $s=3p/2$ we get
% \begin{equation*}
%   \|\partial_t\nabla\bfu\|_{\frac{6p}{4+p}}^2\leq c\|\partial_t\bF(\bD\bu)\|_2^2(\delta+\|\nabla\bu\|_{3p/2})^{2-p}
% \end{equation*}
% and a similar argument again using the same lemma for second order derivatives implies
Thus, we also get 
\begin{align}
  \label{eq:regularity}
  \begin{aligned}
    \bfu &\in L^2(I;W^{2,\frac{6p}{4+p}}(\Omega)),
    \\
    \frac {\partial\bfu}{\partial t} &\in L^2(I;W^{1,\frac{6p}{4+p}}(\Omega)),
  \end{aligned}
\end{align}
where the bounds depend only on $\|\bF(\bD\bu)\|_{W^{1,2}(I\times
  \Omega)}$ and $\delta_0$. This implies in particular that $\bfu \in
C(\overline {I};W^{1,\frac{6p}{4+p}}(\Omega))$. 
\section{On the numerical error}
In this section we prove the error estimates from
Theorem~\ref{thm:theorem-parabolic}. To this end we need to derive the equation for the error and to use the discrete Gronwall
lemma together with approximation properties coming from the fact that we have regular
enough solutions, together with the assumption on the nonlinear operator $\bS$.
\subsection{Approximation properties}
Crucial properties to estimate the quasi-norm of the finite dimensional projections
concern the time regularity of the continuous solution. In particular,
the last term in the estimate from Proposition~\ref{prop:Ph} (ii) for ${\bv =\bv(t)}$, $\bw=\bv(s)$ will give convergence rates with respect
to time, under appropriate regularity assumptions on the partial derivative with respect
to time. This is based on the following lemma which is in the same spirit
as~\cite[Proposition~3.6]{bdr-7-5-num} and which will be used several times in the sequel.
\begin{lemma}
  \label{lem:Bochner-lemma}
  Let be given $f:I\to X$, where $X$ is a Banach space and let $f$ be strongly
  measurable. Let us assume that
  \begin{equation*}
    f, \frac{\partial f}{\partial t}\in L^2(I;X).
  \end{equation*}
%\marginpar{Here wrong sum starting from  0}Then, it follows that 
Then, it holds 
\begin{equation}
  \label{eq:estimate-time-derivative}
  \begin{aligned}
     \ksumo\;\Int\Int\|f(s)-f(t)\|_X^2\,ds\, dt &\leq
    \kappa^2\Bignorm{\frac{\partial f}{\partial t}}_{L^2(I;X)}^2,
    \\
     \ksumo\;\Int\|f(s)-f(t_{m})\|_X^2\,ds &\leq \kappa^2\Bignorm{\frac{\partial f}{\partial t}}_{L^2(I;X)}^2.
  \end{aligned}
\end{equation}
\end{lemma}
\begin{proof}
  We prove the first estimate from~\eqref{eq:estimate-time-derivative}. We start, thanks
  to the Bochner theorem (see Yosida~\cite[Chap.~V.5]{yosida}), by estimating the
  difference as follows
\begin{equation*}
  \|f(s)-f(t)\|_X=\Big\|\int\limits_{t}^s \frac{\partial f}{\partial t} (\rho)\,d\rho\Big\|_X
  \leq\int\limits_{t}^s \Big\|\frac{\partial f}{\partial t} (\rho)\Big\|_X\,d\rho\qquad
  \forall\,s,t\in  I.
\end{equation*}
Hence, by Cauchy-Schwarz inequality, for all $s,t\in I_m\subseteq I$, since
$|s-t|\leq\kappa$, it follows
\begin{equation*}
  \begin{aligned}
    \|f(s)-f(t)\|_X^2&\leq\Big(\int\limits_{t}^s
    \Big\|\frac{\partial f}{\partial t}(\rho)\Big\|_X\,d\rho\Big)^2\leq \int\limits_{t}^s
    \Big\|\frac{\partial f}{\partial t} (\rho)\Big\|_X^2 d\rho \  \Big|\int\limits_{t}^s\,d\rho\Big|,
    \\
    &\leq    \kappa \int\limits_{t}^s \Big\|\frac{\partial f}{\partial t}(\rho)\Big\|_X^2\,d\rho.
\end{aligned}
\end{equation*}
From the latter, we get by a double integration 
\begin{equation*}
  \Int\Int\|f(s)-f(t)\|_X^2\,ds \,dt\leq \kappa\,   \Int\Int\int\limits_{I_m}
  \Big\|\frac{\partial f}{\partial t} (\rho)\Big\|_X^2\,d\rho\, ds\, dt=\kappa\int\limits_{I_m}
  \Big\|\frac{\partial f}{\partial t} (\rho)\Big\|_X^2\,d\rho,
\end{equation*}
since the right-hand side does not depend on $s$ and
$t$.
Multiplying by $\kappa$ and summing over $m$ we get
\begin{equation*}
  \ksumo\Int\Int\|f(s)-f(t)\|_X^2\,ds\, dt \leq
  \kappa^2\ksumo\int\limits_{I_m}
  \Big\|\frac{\partial f}{\partial t} (\rho)\Big\|_X^2\,d\rho=\kappa^2\Bignorm{\frac{\partial f}{\partial t}}_{L^2(I;X)}^2.
\end{equation*}

The second estimates follows in the same way observing that $f,\frac
{\partial f}{\partial t} \in L^2(I;X)$
implies $f\in C(I;X)$, hence $f(t_{m})$ is well defined. Then, one can re-write the
difference term as follows
  \begin{equation*}
    \|f(s)-f(t_{m})\|_X=\Big\|\int\limits_{t_m}^s \frac{\partial f}{\partial t} (\rho)\,d\rho\Big\|_X,
  \end{equation*}
and by using the same techniques as before one concludes the proof.
\end{proof}

%This lemma will be often used for $f =\bF(\bD\bu)$ and thus yields
%convergence rates in time. 
%
%\section{Error estimate}
%
% we subtract the equation~\eqref{eq:cont-var} from~\eqref{eq:fdiscr-var2} to
%he equation for the error
%quation}
%{eq:error-var2}
%rac{d}{dt}(\hbfuMh(t)-\bfu(t)),\bfvM)+\langle A\obfuMh(t)-A\bfu(t),\bfvM\rangle=\langle\obffMh(t)-\bff(t),\bfvM\rangle,
%ation}
%r all $ t,$ and for all $\bfvM\in V_M$. In order to get the estimate for the error
%s test function
%quation*}
%\obfuMh(t)-P_M\bfu(t)
%ation*}
%
%NO}
%
%
%
%
%
%\comment{replace $\overline {u}^m$ by $\Int u(s) \, ds$: L. it seems now done}
\subsection{Error estimates}
To prove the Theorem~\ref{thm:theorem-parabolic} we take the retarded
averages of~\eqref{eq:cont-var} over $I_m$, $m=1,\ldots, M$
%\marginpar{which version you prefer?}
\begin{equation}
  \label{eq:avg-cont-var}
    ( d_t\bfu(t_{m}),\bfv)+\Int\hskp{
      \bS(\bD\bfu(s))}{\bD\bfv} \,ds=\Int (\bff(s),\bfv )\, ds,
\end{equation}
which is valid for all $\bfv\in V$. As usual, we subtract equation~\eqref{eq:avg-cont-var}
from~\eqref{eq:fdiscr-var} to obtain the equation for the error
\begin{equation}
  \label{eq:error-var}
    ( {d_t}(\bfuMh-\bfu(t_{m})),\bfvM)+\Int\hskp{
      \bS(\bD\bfuMh)-\bS(\bD\bfu(s))}{\bD\bfv} \,ds =\Int (\ff (t_m)-\bff(s),\bfv )\, ds,
\end{equation}
valid for all $ m=1,\ldots, M$ and for all $\bfvM\in V_h$.  

%\section{Proof of the main result}

%In order to get the estimate for the error we use 
%and we  prove the following inequality
\begin{proposition}
  \label{prop:err}
Under the assumptions of Theorem~\ref{thm:theorem-parabolic} we have
the following discrete inequality, valid for $m=1,\dots,M$ and
$0<\kappa\le 1 $
\begin{equation}
  \label{eq:error-estimate}
  \begin{aligned}
    &{d_t}\|\bfuMh-\bfu(t_{m})\|_2^2+c\,\|\bF(\bD\bfuMh)-\bF(\bD\bfu(t_{m}))\|^2_2
    \\
    &\leq c\,h^{2}\Int\|\nabla\bF(\bD\bu(s))\|_2^{2}\,ds +
    c\Int\|\bF(\bD\bu(s))-\bF(\bD\bu(t_{m}))\|^2_2\,ds
    \\
    &\quad+ c\,\frac {h^{2+4/p'}}\kappa \Int \| \nabla ^2\bu(s)\|_{\frac{6p}{4+p}} ^2\,ds
    + c\,\frac{h^{4/p'}}{\kappa}\Big\|\nabla \bfu(t_{m})-\Int
    \nabla\bfu(s)\,ds\Big\|_{\frac{6p}{4+p}}^2
    \\
    &\quad + c\, \Big\|\ff(t_{m})-\Int \ff(s)\,ds\Big\|_{2}^2 + c \, \|\bfuMh
    -\bu(t_{m})\|^2_2.
    \end{aligned}
\end{equation}
\end{proposition}
To prove Proposition~\ref{prop:err} we treat separately the terms resulting from
using in~\eqref{eq:error-var} the legitimate test function 
\begin{equation*}
  \bfvM=\bfuMh-P_h\bfu(t_{m})\,.%=\bfuMh-\bfu(t_{m})+\bfu(t_{m})-P_h\bfu(t_{m})
\end{equation*}
We start with the term involving the discrete time-derivative.
\begin{lemma}
  \label{lem:time}
  It holds that 
  \begin{equation*}
    \begin{aligned}
      &({d_t}(\bfuMh-\bfu(t_{m})),\bfuMh-P_h\bfu(t_{m}))
      \\
      &\geq\frac{1}{2}d_t\|\bfuMh-\bfu(t_{m})\|^2_2 
      -c\, \frac{h^{2+4/p'}}{\kappa}\Int \|\nabla ^2\bu(s)\|_{\frac{6p}{4+p}}^2\,d s
      \\
      &\quad -c\,\frac{h^{4/p'}}{\kappa}\Big\|\nabla\bfu(t_{m})-\Int
      \nabla \bfu(s)\,ds\Big\|_{\frac{6p}{4+p}}^2.
    \end{aligned}
  \end{equation*}
\end{lemma}
\begin{proof}
  We re-write in this case the test function as follows
  \begin{equation}\label{eq:split}
    \bfvM=\bfuMh-\bfu(t_{m})+\bfu(t_{m})-P_h\bfu(t_{m}),
  \end{equation}
  to obtain
  \begin{equation*}
    \big({d_t}(\bfuMh-\bfu(t_{m})),\bfuMh-\bfu(t_{m})\big)=\frac{1}{2}d_t\|\bfuMh-\bfu(t_{m})\|^2_2
    +\frac{\kappa}{2}\|d_t(\bfuMh-\bfu(t_{m}))\|^2_2.  
  \end{equation*}
  The remaining term is treated as follows
  \begin{equation*}
    (d_t(\bfuMh-\bfu(t_{m})),\bfu(t_{m})-P_h\bfu(t_{m}))\leq\frac{\kappa}{4}\|d_t(\bfuMh-\bfu(t_{m}))\|^2_2
    +\frac{1}{\kappa}\|\bfu(t_{m})-P_h\bfu(t_{m})\|^2_2. 
  \end{equation*}
  The first term is absorbed in the last term of the previous equality. Note that the
  second term can not be estimated as $\|\bfu(t_m)- P_h\bfu(t_m)\|^2_2\leq c
  \,h^{2+4/p'}\|\nabla ^2\bu(t_m)\|_{\frac{6p}{4+p}}^2$, since the right-hand side might
  be infinite; in view of the regularity of $\bu$ the norm
  $\|\nabla^2\bu(t)\|_{\frac{6p}{4+p}}$ is only finite for almost everywhere $t\in I$. Thus, we
  proceed differently and add and subtract (time) mean values. Using that $P_h
  \dashint\nolimits_{I_m}\!\!\bfu(s)\,ds =\dashint\nolimits_{I_m} \!\!P_h \bfu(s)\,ds$, and 
  Fubini's theorem, we get
  \begin{equation*}
    \begin{aligned}
      &\|\bfu(t_{m})-P_h\bfu(t_{m})\|_2
      \\
%      &=\Big\|\bfu(t_{m})-\Int \bfu\, ds
%      -P_h\Big(\bfu(t_{m})-\Int\bfu\,ds\Big)+\Int (\bfu-
%      P_h\bfu)\,ds\Big\|_2
%      \\
      &\leq \Big\|\bfu(t_{m})-\Int \bfu(s)\, ds
      -P_h\Big(\bfu(t_{m})-\Int\bfu(s)\,ds \Big)\Big\|_2+ \Big\|\Int
      \bfu(s)- P_h\bfu(s)\,ds\Big\|_2.
    \end{aligned}
  \end{equation*}
  Both terms are estimated using Proposition \ref{prop:Ph} (ii) and the
  regularities~\eqref{eq:regularity} to obtain
%\marginpar{NOO}
%  \begin{equation*}
%    \Big\|\bfu(t_{m})-\Int
%    \bfu-P_h\Big(\bfu(t_{m})-\Int\bfu\Big)\Big\|^2\leq c \,h^2 \Big\|\nabla \Big(\bfu(t_{m})-\Int
%    \bfu\Big)\Big\|^2.
%  \end{equation*}
  \begin{align*}
    \Big\|\bfu(t_{m})-\Int
    \bfu(s)\,ds-P_h\Big(\bfu(t_{m})-\Int\bfu(s)\,ds\Big)\Big\|_2
    &\le c\, h^{2/p'}\big\|\nabla\bfu(t_{m})-\Int \nabla\bfu(s)\,ds \big\|_{\frac{6p}{4+p}}\,,
  \end{align*}
  where we used that $W^{1,\frac{6p}{4+p}}(\Omega)\vnor \vnor L^2(\Omega)$, valid for all
  $p>1$, and
%  note that here the restriction $p>4/3$ occurs, and
%  For the second term from the rhs we can use the following estimates coming from the
%  properties of the interpolation
%  and~\eqref{eq:sobolev-interpolation}-\eqref{eq:regularity2} to obtain
  \begin{equation*}
    \begin{aligned}
      \Big\|\Int \bfu(s)- P_h\bfu(s)\,ds\Big\|^2_2&=\int\limits\limits_\Omega\Big|\Int \bfu(s)-
      P_h\bfu(s)\,ds\Big|^2\,d\bx\leq\int\limits_\Omega\Int |\bfu(s)- P_h\bfu(s)|^2\,ds\,d\bx
      \\
      &=\Int\;\|\bfu- P_h\bfu\|_2^2\, ds\leq c
      \,h^{2+4/p'}\Int       \|\nabla ^2\bu(s)\|^2_{\frac {6p}{p+4}}\,ds.
%      \\
%      &\leq c \frac{h^{2+4/p'}}{\kappa}\int\limits_{t_{m-1}}^{t_{m}}\|\nabla^{2}\bu\|_{\frac{6p}{4+p}}^2\,ds
    \end{aligned}
  \end{equation*}
%
%
%  This term must then summed over $h=1,\dots,N$, obtaining
%  \begin{equation*}
%    \frac{c \,h^2}{\kappa}\ksum   \Big\|\nabla \Big(\bfu(t_{m})-\Int
%    \bfu\Big)\Big\|^2\leq\frac{c \,h^2}{\kappa}\ksum   \Big\|\nabla \Big(\bfu(t_{m})-\Int
%    \bfu\Big)\Big\|^2_{H^1}\leq\frac{c \,h^2}{\kappa}\|\partial_t\nabla \bfu\|^2.
%  \end{equation*}
%%
  Putting the estimates together we obtain the assertion.
\end{proof}
%\vspace{3cm}
%
%
%\bigskip
Next, we estimate the  term with the $(p,\delta)$-structure and obtain  the
following inequality: 
\begin{lemma}
  \label{lem:ell}
  It holds that 
  \begin{equation*}
    \begin{aligned}
      &\Int\bighskp{ \bS(\bD\bfuMh)-\bS(\bD\bfu(s))}{\bD\bfuMh -\bD
        P_h\bfu(t_{m})} \,ds
      \\
      &\geq\|\bF(\bD\bfuMh)-\bF(\bD\bu(t_m))\|^{2}_2
      -c\, h^{2}\Int\|\nabla\bF(\bD\bu(s))\|^{2}_2\,ds \\
      &\quad - c\Int\|\bF(\bD\bu(s))-\bF(\bD\bu(t_{m}))\|^2_2\,ds.
    \end{aligned}
  \end{equation*}
\end{lemma}
\begin{proof}
We   re-write the term with the $p$-structure as follows
  \begin{equation*}
    \begin{aligned}
      &\Int\bighskp{ \bS(\bD\bfuMh)-\bS(\bD\bfu(t_m))+
        \bS(\bD\bfu(t_m))-\bS(\bD\bfu(s))}{\bD\bfuMh -\bD
        P_h\bfu(t_{m})} \,ds
      \\
      &= \Int\bighskp{ \bS(\bD\bfuMh)-\bS(\bD\bfu(t_m))}{\bD\bfuMh
        -\bD P_h\bfu(t_{m})} \,ds
      \\
      &\quad + \Int\bighskp{
        \bS(\bD\bfu(t_m))-\bS(\bD\bfu(s))}{\bD\bfuMh -\bD
        P_h\bfu(t_{m})} \,ds =:\mathcal{A}_1%A_{2}
      +\mathcal{A}_2,%A_1
    \end{aligned}
  \end{equation*}
  and estimate the two terms separately.

  {\bf Estimate of $\mathcal{A}_1%A_{2}
    $}: We have 
  \begin{equation*}
    \begin{aligned}
      \mathcal{A}_1%A_{2}  
%      &= \langle A\bfuMh-A\bfu(t_{m}),\bfuMh-\bfu(t_{m})+\bfu(t_{m})-P_h\bfu(t_{m})\rangle
%      \\
      &=\bighskp{ \bS(\bD\bfuMh)-\bS(\bD\bfu(t_m))}{\bD\bfuMh
        -\bD \bfu(t_{m})}
      \\
      &\quad + \bighskp{ \bS(\bD\bfuMh)-\bS(\bD\bfu(t_m))}{\bD\bfu(t_m)
        -\bD P_h\bfu(t_{m})}
      =:\mathcal{A}_{1,1}+\mathcal{A}_{1,2}.
    \end{aligned}
  \end{equation*}
  The first term is giving the information
  \begin{equation*}
    \mathcal{A}_{1,1}= \bighskp{ \bS(\bD\bfuMh)-\bS(\bD\bfu(t_m))}{\bD\bfuMh
        -\bD \bfu(t_{m})}\geq
    c\,\|\bF(\bD\bfuMh)-\bF(\bD\bfu(t_{m}))\|^2_2,
  \end{equation*}
  while the second can be estimated as follows, by adding and subtracting the average
  $\dashint\nolimits_{I_m}\!\!\bfu(s)\,ds$ in the second entry. In fact, we have
  \begin{equation*}
    \begin{aligned}
      \mathcal{A}_{1,2} %&= \langle A\bfuMh-A\bfu(t_{m}),\bfu(t_{m})-P_h\bfu(t_{m})\rangle
     % \\
      & = \Int\bighskp{ \bS(\bD\bfuMh)-\bS(\bD\bfu(t_m))}{\bD\bfu(s) -\bD
        P_h\bfu(t_{m})} \,ds
      \\
      &\quad + \Int\bighskp{ \bS(\bD\bfuMh)-\bS(\bD\bfu(t_m))}{\bD\bfu(t_m) -\bD
        \bfu(s)} \,ds=:\mathcal{B}_1+\mathcal{B}_2.
    \end{aligned}
  \end{equation*}
  By Proposition~\ref{lem:hammer} it follows %\marginpar{p8}
  \begin{equation*}
    \begin{aligned}
      |\mathcal{B}_2|&
      \leq
      \epsilon\,\|\bF(\bD\bu(t_{m}))-\bF(\bD\bu_h^m)\|^2_2+c_\epsilon\Int
      \|{\bF(\bD\bu(s))}-\bF(\bD\bu(t_{m}))\|^2_2\,ds.
    \end{aligned}
  \end{equation*}
  Next, we split $\mathcal{B}_{1}$ as follows, by adding and subtracting $P_h
  \dashint\nolimits_{I_m}\!\!\bfu(s)\,ds =\dashint\nolimits_{I_m} \!\!P_h \bfu(s)\,ds$, again in
  the second entry, %\marginpar{p9}
  \begin{equation*}
    \begin{aligned}
      \mathcal{B}_1 
      &= \Int\bighskp{ \bS(\bD\bfuMh)-\bS(\bD\bfu(t_m))}{\bD\bfu(s) -\bD
        P_h\bfu(s)} \,ds
      \\
      &\quad +\Int\bighskp{ \bS(\bD\bfuMh)-\bS(\bD\bfu(t_m))}{\bD P_h\bfu(s) -\bD
        P_h\bfu(t_{m})} \,ds 
      =:\mathcal{C}_1+\mathcal{C}_2.
    \end{aligned}
  \end{equation*}
  The term $\mathcal{C}_2$ can be estimated by using
  Proposition~\ref{lem:hammer} (i) and Young's
  inequality~\eqref{ineq:young} as 
  \begin{equation*}
    \begin{aligned}
      \mathcal{C}_2 % &= \Int \langle A\bfuMh-A\bfu(t_{m}),P_h[{\bu}(s)-\bfu(t_{m})]\rangle\,ds
      %\\
      &\leq c\Int\int\limits_\Omega\phi'_{|\bD\bfu(t_{m})|}\big(|\bD\bfuMh-\bD\bfu(t_{m})|\big)\,
      \big|\bD(P_h[{\bu}(s)-\bfu(t_{m})])\big|\,d\bx \,ds
      \\
      &\leq \epsilon
      \Int\int\limits_\Omega\phi_{|\bD\bfu(t_{m})|}\big(|\bD\bfuMh-\bD\bfu(t_{m})|\big)\,d\bx
      \,ds
      \\
      &\quad +
      c_\epsilon\Int\int\limits_\Omega\phi_{|\bD\bfu(t_{m})|}\big(\big|\bD(P_h[{\bu}(s)-\bfu(t_{m})])\big|\big)\,d\bx\,
      ds
      \\
      &\leq \epsilon \, \|\bF(\bD\bfuMh)-\bF(\bD\bfu(t_{m})\|^2_2
      +c_\epsilon\Int\int\limits_\Omega\phi_{|\bD\bfu(t_{m})|}\big(\big|\bD(P_h[{\bu}(s)-\bfu(t_{m})])\big|\big)\,d\bx\,
      ds.
    \end{aligned}
  \end{equation*}
  The latter term from the above inequality can be estimated by a shift change, see
  Proposition~\ref{lem:hammer} (ii)
  \begin{equation*}
    \begin{aligned}
      &\int\limits_\Omega\phi_{|\bD\bfu(t_{m})|}\big(\big|\bD(P_h[{\bu}(s)-\bfu(t_{m})])\big|\big)\,d\bx
      \\
      &\leq c \,\|\bF(\bD\bfu(t_{m})-\bF(\bD\bfu(s))\|^2_2 +c \int\limits_\Omega
      \phi_{|\bD\bfu(s)|}\big(\big|\bD(P_h[{\bu}(s)-\bfu(t_{m})])\big|\big)\,d\bx.
    \end{aligned}
  \end{equation*}
For the last term we use Proposition~\ref{prop:Ph} (iii) and obtain 
\begin{equation*}
  \begin{aligned}
    |\mathcal{C}_{2}|&\leq \epsilon
    \,\|\bF(\bD\bfuMh)-\bF(\bD\bfu(t_{m})\|^2_2
    +c_\epsilon \Int\|\bF(\bD\bfu(t_{m})-\bF(\bD\bfu(s))\|^2_2\,ds
    \\
    &\quad + c_\epsilon\,h^{2}\Int\|\nabla\bF(\bD\bu(s))\|^2_2\,ds.
  \end{aligned}
\end{equation*}
We estimate now $\mathcal{C}_{1}$ by adding and subtracting $\bS(\bD\bu(s))$ in the first
entry
%\marginpar{13}
and get
\begin{equation*}
  \begin{aligned}
    \mathcal{C}_1% &=\Int\bighskp{ \bS(\bD\bfuMh)-\bS(\bD\bfu(t_m))}{\bD\bfu(s) -\bD
    %     P_h\bfu(s)} \,ds
    % \\
    &= \Int\bighskp{ \bS(\bD\bfu(s))-\bS(\bD\bfu(t_m))}{\bD\bfu(s) -\bD
      P_h\bfu(s)} \,ds
    \\
    &\quad + \Int\bighskp{ \bS(\bD\bfuMh)-\bS(\bD\bfu(s))}{\bD\bfu(s) -\bD
      P_h\bfu(s)} \,ds
    =:\mathcal{D}_1+\mathcal{D}_2.
  \end{aligned}
\end{equation*}
Then, by Proposition~\ref{lem:hammer} (ii) and Proposition~\ref{prop:Ph} (i) it follows 
\begin{equation*}
  \begin{aligned}
    | \mathcal{D}_1|%&=\Big|\Int \langle A\bfu(s)-A\bfu(t_{m}),\bu(s)-P_h\bu(s)\rangle\,ds\Big|
    %\\
    &\leq c\Int\|\bF(\bD\bu(s))-\bF(\bD\bu(t_{m}))\|^2_2\,ds+c\Int\|\bF(\bD
    P_h\bu(s))-\bF(\bD\bu(s))\|^2_2\,ds
    \\
    &\leq c\Int\|\bF(\bD\bu(s))-\bF(\bD\bu(t_{m}))\|^2_2\,ds
    +c\,h^{2}\Int\|\nabla\bF(\bD\bu(s))\|^2_2\,ds. 
  \end{aligned}
\end{equation*}
The other term $\mathcal{D}_2$ is estimated in the following manner, by Proposition~\ref{lem:hammer}
(ii), by adding and subtracting $\bF(\bD\bu(t_{m}))$, and Proposition~\ref{prop:Ph} (i) 
\begin{equation*}
  \begin{aligned}
    |\mathcal{D}_2|%&=\Big|\Int\langle
    %A\bfuMh-A\bfu(s),\bu(s)-P_h\bu(s)\rangle\,ds \Big|
    %\\
    &\leq \epsilon\Int\|\bF(\bD\bu(s))-\bF(\bD\bfuMh)\|^2_2\,ds+
    c_\epsilon\Int\|\bF(\bD P_h\bu(s))-\bF(\bD\bu(s))\|^2_2\,ds
    \\
    &\leq \epsilon\,\|\bF(\bD\bfuMh)-\bF(\bD\bu(t_{m}))\|^2_2+
    \epsilon\Int\|\bF(\bD\bu(s))-\bF(\bD\bu(t_{m}))\|^2_2\,ds
    \\
    &\qquad+c_\epsilon
    h^2\Int\|\nabla\bF(\bD\bu(s))\|^2_2\,ds.
  \end{aligned}
\end{equation*}
{\bf Estimate of $\mathcal{A}_2$}:
%\marginpar{p14-15}
%
We now estimate the term $\mathcal{A}_2$, first by adding and subtracting
$\bD P_h\bu(s)$ in the second entry to get
\begin{equation*}
  \begin{aligned}
    \mathcal{A}_2 %&= \Int \langle A\bfu(t_{m})-A\bu(s),\bfuMh-P_h\bfu(t_{m})\rangle
%    \\
    &= \Int\bighskp{ \bS(\bD\bfu(t_m))-\bS(\bD\bfu(s))}{\bD P_h\bfu(s) -\bD
      P_h\bfu(t_m)} \,ds
    \\
    &\quad + \Int\bighskp{ \bS(\bD\bfu(t_m))-\bS(\bD\bfu(s))}{\bD\bfuMh -\bD
      P_h\bfu(s)} \,ds
    =:\mathcal{E}_{1}+\mathcal{E}_{2}.
  \end{aligned}
\end{equation*}
The term $\mathcal{E}_{1}$ is estimated using Young's inequality, Proposition~\ref{lem:hammer} and 
Proposition~\ref{prop:Ph} (iii) 
\begin{equation*}
  \begin{aligned}
    |\mathcal{E}_{1}|&\le c\Int\int\limits_{\Omega}\phi'_{|\bD\bu(s)|}\Big(\big|\bD\bfu(t_{m})-\bD\bu(s)\big|\Big)\big|\bD
    P_h\bu(s)-\bD P_h\bfu(t_{m})\big|\,d\bx \,ds
    \\
    &\leq c \Int\int\limits_{\Omega}\phi_{|\bD\bu(s)|}
    \Big(\big|\bD\bfu(t_{m})-\bD\bu(s)\big|\Big)\,d\bx\, 
    ds
    \\
    &\quad +c \Int\int\limits_{\Omega}\phi_{|\bD\bu(s)|}\Big(\big| \bD
    P_h\bu(s)-\bD P_h\bfu(t_{m}) \big|\Big)\,{\rm d\bx \,ds}
    \\
    &\leq %\epsilon\Int\|\bF(\bD\bu(t_{m}))-\bF(\bD\bu(s))\|^{2}_2\,ds
    c\Int\|\bF(\bD\bu(t_{m}))-\bF(\bD\bu(s))\|^2_2\,ds 
    % \\
    % &\quad
    +c\, h^2\Int\|\nabla\bF(\bD\bu(s))\|^2_2\,ds.
  \end{aligned}
\end{equation*}
The term $\mathcal{E}_{2}$ is estimated by adding and subtracting
$\bD\bu(s)$ and $\bD\bu(t_{m})$ in the second entry to get 
\begin{equation*}
  \begin{aligned}
    \mathcal{E}_{2}%&=    \Int\langle
   % A\bfu(t_{m})-A\bu(s),\bfuMh-P_h\bfu(s)\rangle \,ds
   % \\
   & = \Int\bighskp{ \bS(\bD\bfu(t_m))-\bS(\bD\bfu(s))}{\bD \bfu(s) -\bD
      P_h\bfu(s)} \,ds
    \\
    &\quad + \Int\bighskp{ \bS(\bD\bfu(t_m))-\bS(\bD\bfu(s))}{\bD\bfu(t_m) -\bD
      \bfu(s)} \,ds
    \\
    &\quad + \Int\bighskp{ \bS(\bD\bfu(t_m))-\bS(\bD\bfu(s))}{\bD\bfuMh -\bD
      \bfu(t_m)} \,ds
    =:\mathcal{E}_{2,1}+\mathcal{E}_{2,2}+\mathcal{E}_{2,3}.
  \end{aligned}
\end{equation*}
Then, Proposition~\ref{lem:hammer} and Proposition~\ref{prop:Ph} (i)  yield 
\begin{equation*}
  \begin{aligned}
    |\mathcal{E}_{2,1}| %&\le c\Int\int\limits_{\Omega}\phi'_{|\bD\bu(s)|}\Big(\big|\bD\bfu(t_{m})-\bD\bu(s)\big|\Big)
    %\big|\bD\bu(s)-\bD P_h\bfu(s)\big|\,d\bx ds
    %\\
    &\leq c \Int\|\bF(\bD\bu(s))-\bF(\bD\bu(t_{m}))\|^{2}_2 \,ds+c \Int\|
    \bF(\bD\bu(s))-\bF(\bD P_h\bfu(s))\|^{2}_2 \,ds
    \\
    &\leq c \Int\|\bF(\bD\bu(s))-\bF(\bD\bu(t_{m}))\|_2^{2}
    \,ds +c\, h^{2}\Int\|\nabla\bF(\bD\bu(s))\|^{2}_2 \,ds,
  \end{aligned}
\end{equation*}
as well as 
\begin{equation*}
  \mathcal{E}_{2,2}\le c \Int\|\bF(\bD\bu(s))-\bF(\bD\bu(t_{m}))\|^{2}_2\,ds,
\end{equation*}
and 
\begin{equation*}
  \begin{aligned}
    |\mathcal{E}_{2,3}| %&\le c \Int\int\limits_{\Omega}\phi'_{|\bD\bu(t_{m})|}\Big(\big|\bD\bfu(t_{m})-\bD\bu(s)\big|\Big)
    %\big|\bD \bfuMh-\bD\bu(t_{m})\big|\,d\bx\, ds
    %\\
%    &\leq c \Int\|\bF(\bD\bu(s))-\bF(\bD\bu(t_{m}))\|^{2}+c \Int\|
%    \bF(\bD\bu(s))-\bF(\bD P_h\bfu(s))\|^{2}
%    \\
    &\leq \epsilon \, \|\bF(\bD\bfuMh)-\bF(\bD\bu(t_{m}))\|^{2}_2 + c_{\epsilon}\Int\|\bF(\bD(s))-\bF(\bD\bu(t_{m}))\|^{2}_2 \,ds.
  \end{aligned}
\end{equation*}
Putting all these estimate together and choosing $\vep$ small enough,
we arrive at the estimate in Lemma~\ref{lem:ell}.
\end{proof}

\begin{lemma}\label{lem:f}
  It holds that
\begin{align*}
  &\Bigabs{ \;\Int \big(\ff(t_m) - {\bff}(s),\bfuMh-P_h\bu(t_m)\big )\,ds}
  \\
  &\le c\,\Big \|\ff(t_m) - \Int {\bff}\, ds \Big\|_2^2 + c\,
    \|\bfuMh-\bu(t_m)\|_2^2
    \\
    &\quad +c\,h^{2+4/p'} \Int \|\nabla ^2 \bu(s)\|_{\frac{6p}{p+4}}
    ^2\,ds + c\,h^{4/p'}\Big\|\bfu(t_{m})-\Int
    \bfu(s)\,ds\Big\|_{1,\frac{6p}{4+p}}^2.
\end{align*}
\end{lemma}
\begin{proof}
  Using the splitting~\eqref{eq:split} and Young's inequality we get
\begin{align*}
  &\Bigabs{ \;\Int \big(\ff(t_m) - {\bff}(s),\bfuMh-P_h\bu(t_m)\big )\,ds}
  \\
  &\le c\,\Big \| \ff(t_m) - \Int {\bff}(s)\, ds \Big\|_2^2 + c\,
    \|\bfuMh-\bu(t_m)\|_2^2 + c\, \|\bu(t_m)-P_h\bu(t_m)\|_2^2\,.
\end{align*}
The last term was already treated in the proof of
Lemma~\ref{lem:time}. There we proved 
\begin{align*}
  &\|\bu(t_m)-P_h\bu(t_m)\|_2^2
  \\
  &\le c\,h^{2+4/p'} \Int \|\nabla ^2 \bu(s)\|_{\frac{6p}{p+4}}
    ^2\,ds + c\,h^{4/p'}\Big\|\bfu(t_{m})-\Int
    \bfu(s)\,ds\Big\|_{1,\frac{6p}{4+p}}^2\,,
\end{align*}
which yields the assertion.
\end{proof}

\begin{proof}[Proof of Proposition~\ref{prop:err}]
  The assertion follows from Lemma~\ref{lem:time}, Lemma~\ref{lem:ell}
  and Lemma~\eqref{lem:f}.
\end{proof}

\begin{proof}[Proof of Theorem~\ref{thm:theorem-parabolic}]
  We now prove the main result. Multiplying~\eqref{eq:error-estimate} by $\kappa$ and
  summing over $m=1,\dots,N$, for $N\le M$, we get
  \begin{equation*}
    \begin{aligned}
      % \max_{0\leq m\leq N}
      &\|\bfu^N_h-\bfu(t_{N})\|^2_2+\sum_{m=1}^{N}
      \|\bF(\bD\bfuMh)-\bF(\bD\bfu(t_{m}))\|^2_2
      \\
      &\leq c\, h^{2}\sum_{m=1}^{N}
      \int\limits_{I_m}\|\nabla\bF(\bD\bu(s))\|^{2}_2\,ds + c\,
      \kappa\sum_{m=1}^{N}\Int\|\bF(\bD\bu(s))-\bF(\bD\bu(t_{m}))\|^2_2\,ds 
      \\
    &\quad+
    c\,\frac {h^{2+4/p'}}\kappa \sum_{m=1}^{N}\int\limits_{I_m} \|\nabla ^2
    \bu(s)\|_{\frac{6p}{p+4}} %| \bu|_{{1+2/p',2}}%\frac{6p}{4+p}}
    ^2\,ds + c\,{h^{4/p'}} \sum_{m=1}^{N}\Big\|\bfu(t_{m})-\Int
    \bfu(s)\,ds\Big\|_{1,\frac{6p}{4+p}}^2
    \\
    &\quad  + c\, \kappa \sum_{m=1}^{N} \Big\|\ff(t_{m})-\Int
    \ff(s)\,ds\Big\|_{2}^2 + c \, \kappa \sum_{m=1}^{N}\|\bfuMh -\bu(t_{m})\|^2_2
    + c\, \|\bu_h^{0}-\bu_{0}\|_2^{2}\,.
    \end{aligned}
  \end{equation*}
  First we observe that by condition~\eqref{eq:compatbility}
  \begin{equation*}
    \frac{h^{2+4/p'}}{\kappa}\sum_{m=1}^{N}\int\limits_{I_m}
    \|\nabla ^2 \bu(s)\|_{\frac{6p}{p+4}}
    ^2\,ds\leq\frac{h^{4/p'}}{\kappa}h^{2}\int\limits_{0}^{T} 
    \|\nabla ^2 \bu(s)\|_{\frac{6p}{p+4}}^2\,ds\leq
    c\,h^{2}\int\limits_{0}^{T}\|\nabla ^2
    \bu(s)\|_{\frac{6p}{p+4}}^2\,ds. 
  \end{equation*}
  Next, by using Lemma~\ref{lem:Bochner-lemma} we have
  \begin{equation*}
   \kappa \sum_{m=1}^N\Int\|\bF(\bD\bu(s))-\bF(\bD\bu(t_{m}))\|^2_2\,ds\leq
    k^{2}\int\limits_{0}^{T}\Bignorm{\frac{\partial\bF(\bD\bu)}{\partial
      t}(s)}_2^{2}\,ds,
  \end{equation*}
  and also, by using again~\eqref{eq:compatbility},
  \begin{align*}
    \frac{h^{4/p'}}{\kappa}\kappa\sum_{m=1}^N
    \Big\|\bfu(t_{m})-\Int\bfu(s)\,ds\Big\|^2_{\frac{6p}{4+p}}&\leq
    \frac{h^{4/p'}}{\kappa}\kappa^{2}
    \int\limits_{0}^{T}\Bignorm{\frac{\partial\nabla\bfu}{\partial
        t}(s)}^2_{\frac{6p}{4+p}}\,ds
    \\
    &\leq c \,\kappa^{2}
    \int\limits_{0}^{T}\Bignorm{\frac{\partial\nabla\bfu}{\partial t}(s)}^2_{\frac{6p}{4+p}}\,ds.
  \end{align*}
  Moreover, Lemma~\ref{lem:Bochner-lemma} also yields
  \begin{align*}
    \kappa \sum_{m=1}^{N} \Big\|\ff(t_{m})-\Int
    \ff(s)\,ds\Big\|_{2}^2 \le c \, \kappa^{2}
    \int\limits_{0}^{T}\Bignorm{\frac{\partial\ff}{\partial t}(s)}^2_{2}\,ds.
  \end{align*}
  Hence we have, by using~\eqref{eq:regularity} and the fact that
  $\bu_h^0$ is the $L^2$-projection of $\bu_0$ %and~\eqref{eq:regularity2}
  \begin{equation*}
    \begin{aligned}
      % \max_{0\leq m\leq N}
      &\|\bfu^N_h-\bfu(t_{N})\|^2_2+\sum_{m=1}^N
      \|\bF(\bD\bfuMh)-\bF(\bD\bfu(t_{m}))\|^2_2
      \\
      &\leq c(h^{2} \!+ \!\kappa^{2})\!
      \int\limits_{0}^{T}\!\|\nabla\bF(\bD\bu(s))\|_2^{2}+
      \Bignorm{\frac{\partial\bF(\bD\bu)}{\partial
      t}(s)}_2^{2} + \Bignorm{\frac{\partial\ff}{\partial t}(s)}^2_{2}+
      \Bignorm{\frac{\partial\nabla\bfu}{\partial t}(s)}^2_{\frac{6p}{4+p}}\!
      +\|\nabla^2\bfu(s)\|^2_{\frac{6p}{4+p}} ds
      \\
      &\quad + c \, \kappa \sum_{m=1}^{N}\|\bfuMh
      -\bu(t_{m})\|^2_2+c\, h^2\|\nabla \bu_0\|_2^{2}\,.
    \end{aligned}
  \end{equation*}
  Then, if $\kappa>0$ is small enough such that $c\,\kappa<1$, we can absorb the last
  addendum in the sum from the right-hand side and obtain (using the regularity of $\bu$
  to bound the time integrals in the above formula)
  \begin{equation*}
    \begin{aligned}
      % \max_{0\leq m\leq N}
      \|\bfu^N_h-\bfu(t_{N})\|^2_2+\sum_{m=1}^N
      &\|\bF(\bD\bfuMh)-\bF(\bD\bfu(t_{m}))\|^2_2
      \\
      &
\leq c(h^{2} \!+ \!\kappa^{2})
       +c \, \kappa \sum_{m=1}^{N-1}\|\bfuMh -\bu(t_{m})\|^2_2.%+\|P_{h}\bu_{0}-\bu_{0}\|^{2}.
    \end{aligned}
  \end{equation*}
The discrete Gronwall lemma yields the
assertion.
% \bigskip
% \comment{With the other choice of the initial datum for the discrete problem the calculations are
% the same except that an additional term for $m=0$, hence we get 
%   \begin{equation*}
%     \begin{aligned}
%       % \max_{0\leq m\leq N}
%       &\|\bfu^N_h-\bfu(t_{N})\|^2_2+\sum_{m=1}^N
%       \|\bF(\bD\bfuMh)-\bF(\bD\bfu(t_{m}))\|^2_2
%       \\
%       &
%       \leq c(h^{2} \!+ \!\kappa^{2})
%        +c \, \kappa \sum_{m=1}^{N-1}\|\bfuMh -\bu(t_{m})\|^2_2+\|P_{h}\bu_{0}-\bu_{0}\|^{2}.
%     \end{aligned}
%   \end{equation*}
% and the discrete Gronwall lemma and the assumption $\bu_{0}\in W^{1,2}(\Omega)$ yields the assertion.}
\end{proof}
%
%%%%%%%%%%%%%%%%%%%%%%%%%%%%%%%%%%%%%
%\part{Analysis}
%
 \section{on the existence and uniqueness of regular solutions}
 In this section we prove Theorem~\ref{thm:MT}, i.e., the existence
 and uniqueness of regular solutions of~\eqref{eq:pfluid}, solely
 based on appropriate assumptions on the data. To this end we proceed
 as in~\cite{br-plasticity} and treat a perturbed problem, obtained by
 adding to the tensor field $\bS$ with $(p,\delta)$-structure a linear
 perturbation. We use this approximation to justify the
 computations that follow and to avoid some technical
 problems related with the case $p \in (1,6/5)$ and the lack of an
 evolution triple in this range. From now on we restrict ourselves to the
 case that $\bS$ has $(p,\delta)$-structure some $p \in (1,2]$,
 $\delta \in [0,\infty)$. Let $\bff\in L^{p'}(I\times\Omega)$ and
 $\bu_0\in L^2(\Omega)$ be given.

\subsection{The perturbed problem and some global regularity in the time variable}   
%
%In order to construct the regular solution and to justify some of the following
%computations we proceed as in~\cite{br-plasticity} and consider a perturbed problem,

We have the following result on existence and uniqueness of
time-regular solutions of the perturbed problem.
\begin{proposition}
\label{thm:existence_perturbation}
Let the tensor field $\bS$ have $(p,\delta)$-structure for some $p\in(1,2]$, and
$\delta\in[0,\infty)$ and let $\bff\in L^{p'}(I;L^{p'}(\Omega))\cap
W^{1,2}(I;L^2(\Omega))$ and $\bu_{0}\in W^{1,2}_0(\Omega)\cap W^{2,2}(\Omega)$ with
$\divo(\bS^\epsilon(\bD\bu_0)) \in L^2(\Omega)$ be given. Then,  the perturbed problem
\begin{equation}
  \label{eq:eq-e}
  \begin{aligned}
    \frac{\partial\bue}{\partial t}-\divo \bfS^\vep (\bfD\bue)&=\bff\qquad&&\text{in }I\times\Omega\,,
    \\
    \bue &= \bfzero &&\text{on } I\times\partial \Omega\,,
    \\
    \bue(0)&=\bu_0&&\text{in }\Omega\,,
  \end{aligned}
\end{equation}
where 
\begin{equation*}
%    \label{eq:perturbed_S}
  \bS^\epsilon(\bQ):=\epsilon \,\bQ +
  \bS(\bQ),\qquad\text{with }\epsilon>0\,,
\end{equation*}
possesses a unique
time-regular solution $\bue$% $\bue\in L^2(I;W^{1,2}_0(\Omega))\cap L^\infty(I;L^2(\Omega))$
, i.e., $\bue\in W^{1,\infty}(I;L^{2}(\Omega))\cap
W^{1,2}(I;W^{1,2}_{0}(\Omega))$ with
$\bF(\bD\bue) \in W^{1,2}(I;L^{2}(\Omega))$ satisfies for all 
% a.e.~$t \in I$
$\psi \in C_0^\infty (I)$ and all $\bv\in W^{1,2}_0(\Omega)$
\begin{align}
  \label{eq:weak-eps}
  \int\limits _0^T\Bighskp{\frac {\partial\bfu}{\partial t}(t)
  }{\bfv}\psi (t) \, dt +\int\limits _0^T\hskp{\bS^\vep(\bD\bfu(t))}{\bD\bfv}\psi (t)\,dt =\int\limits
  _0^T\hskp{\bff(t)}{\bfv}\, \psi(t)\, dt\,,
\end{align}
and $\bue (0)=\bu_0$ in $W^{1,2}_0(\Omega)$. 
  % \begin{equation*}
  % \int\limits_0^T    \Big\langle \frac{\partial\bue}{\partial t},\bv \Big \rangle_{W^{1,2}_0(\Omega)}\, dt   +\int\limits_0^T \int\limits_\Omega
  % \bS^\epsilon(\bD\bue)\cdot\bD\bv\,d\bx dt=\int\limits_0^T\int\limits_\Omega \bff\cdot
  % \bv\,d\bx dt\,,
  % \end{equation*} 
  % where $\frac{\partial\bue}{\partial t}\in
  % L^{2}(I;(W^{1,2}_{0}(\Omega))^{*})$.
In addition, the solution $\bue$ satisfies for a.e.~$t \in I$ the
estimates\footnote{Note that
    $c(\delta)$ only indicates that the constant $c$ depends on $\delta$ and will satisfy
    $c(\delta)\le c(\delta_0)$ for all $\delta\le \delta_0$.  } 
\begin{gather}
  \label{eq:main-apriori-estimate2}
  \begin{aligned}
    &\frac{1}{2}\|\bue(t)\|_2^2+\int\limits_0^t \int\limits_\Omega \epsilon\,
    |\nabla\bue|^2+\phi(|\bD\bue|)\,d\bx\,ds
    \\[-2mm]
    &\hspace{2cm}\leq
    \frac{1}{2}\|\bu_0\|_2^2+c(\delta,\Omega)\int\limits_0^T\|\bff(s)\|_{p'}^{p'}\,ds\,,
  \end{aligned}
  \\[2mm]
  \label{eq:estimate-partial-t-F}
  \begin{aligned}
    &\Bignorm{\frac{\partial\bue}{\partial t}(t)}_2^2+\int\limits_0^t\epsilon\,
    \Bignorm{\frac{\partial\bD\bue}{\partial t}(s)}_2^{2}
    +\Bignorm{\frac{\partial\bF(\bD\bue)}{\partial t}(s)}^2_2\,ds
    \\[-1mm]
    &\leq
    C(\delta,\Omega)\Big(\epsilon\|\bu_0\|_{2,2}^{2}+\|\divo(\bS(\bD\bu_0))\|_2^2+\int\limits_0^T
    \|\bff(s)\|^2_2+\Bignorm{\frac{\partial \bff}{\partial
        t}(s)}_2^2\,ds\Big)\,.
  \end{aligned}
\end{gather}
The estimate~\eqref{eq:main-apriori-estimate2} and
\eqref{eq:estimate-partial-t-F} imply that
$\bue\in L^{\infty}(I;L^2(\Omega))$,
$\bF(\bD\bue)\in L^{2}(I\times \Omega)$; and
$\bue\in W^{1,\infty}(I;L^2(\Omega))$,
$\bF(\bD\bue)\in W^{1,2}(I;L^2(\Omega))$, resp., with bounds
independent of $\epsilon>0$.
\end{proposition}
\begin{proof}
  The proof is based on a standard Galerkin approximation. The existence of the Galerkin
  approximations follows from the standard theory of systems of ordinary differential
  equations. Estimate~\eqref{eq:main-apriori-estimate2} is proved on the Galerkin level by
  testing with the Galerkin approximation. Estimate~\eqref{eq:estimate-partial-t-F} is
  obtained by differentiating the Galerkin equations with respect to time and testing with
  the time derivative of the Galerkin approximation. We refer to~\cite{bdr-7-5,dr-7-5} for
  more details. Note that the regularity is enough to justify all calculations and to
  employ the Gronwall lemma to prove uniqueness.
\end{proof}
\begin{remark}\label{rem:weak}
  Note that by the fundamental theorem on the calculus of variations
  the weak formulation \eqref{eq:weak-eps} is equivalent to
\begin{align}\label{eq:weak-eps1}
  \Bighskp{\frac {\partial\bfu}{\partial t}(t)
  }{\bfv} +\hskp{\bS^\vep(\bD\bfu(t))}{\bD\bfv} =
  \hskp{\bff(t)}{\bfv}\qquad \textrm{ a.e.}~t \in I,\  \forall \bv \in W^{1,2}_0(\Omega)\,.
\end{align}  
\end{remark}
In order to prove existence and uniqueness of regular solutions to~\eqref{eq:pfluid}, by
taking the limit $\epsilon\to0^{+}$, we need to prove further regularity for the solution $\bue$,
namely on the second order spatial derivatives.
The regularity in the spatial variables requires an ad hoc treatment (localization) for the
Dirichlet boundary value problem.  To do this we adapt the argument
in~\cite{br-plasticity} for the steady problem, to handle the parabolic problem. We sketch
the relevant steps, pointing out the main new aspects which are present in the
time-dependent case.
\begin{remark}
  In the space periodic case the requested regularity for the spatial derivatives can be
  obtained simply by testing (again the Galerkin approximations) with $-\Delta\bu$, as
  in~\cite{bdr-7-5} to prove for a.e.~$t \in I$ the inequality
  \begin{equation*}
   \frac{1}{2}
   \|\nabla\bue(t)\|_2^{2}+\int\limits_{0}^{t}\epsilon\|\Delta\bue(s)\|_2^{2}+\|\nabla\bF(\bD\bue(s))\|_2^{2}\,ds
   \leq\frac{1}{2}\|\nabla\bu_{0}\|_2^{2}+c\int\limits_{0}^{T}\|\bff(s)\|_{p'}^{p'}\,ds, 
  \end{equation*}
  % \marginpar{on lhs replace T by t?}
  with $c$ depending only on $\delta$ and $\Omega$.
\end{remark}
%
% \begin{remark}
% \label{rem:weak-solutions}
%   We can handle the full range $1<p\le 2$ since we prove existence and uniqueness of
%   rather regular solutions and this makes it possible to properly use the energy estimates.
%   We recall that the study (in three dimensions) of weak solutions for $1<p<6/5$ will
%   require several technical assumptions due to the lack of the standard machinery with
%   evolution triples (cf.~\cite{alex-rose-hirano}).  see~\cite{cc}. ???

%   Moreover, in the case $p\geq6/5$, since $W^{1,p}(\Omega)\subset L^{2}(\Omega)$, then
%   also weak solutions can be easily constructed by the same method. In this case it
%   follows directly that if $\bu\in L^{p}(I,W^{1,p}_{0}(\Omega))$, then $\partial_t\bu\in
%   L^{p'}(I;(W^{1,p}_{0}(\Omega)^{*})$ hence, for all $\bv\in L^p(I;V)$ it follows that
%   \begin{equation*}
% %    \label{eq:weak-formulation3}
%     \int\limits_0^T    \Big \langle \frac {\partial \bu}{\partial t},\bv\Big \rangle\, dt   +\int\limits_0^T \int\limits_\Omega
%     \bS(\bD\bu)\cdot\bD\bv\,d\bx dt=\int\limits_0^T\int\limits_\Omega \bff\cdot
%     \bv\,d\bx dt\,, 
%   \end{equation*}
%   and this allows at least to write properly the error equation, using the finite
%   dimensional projection of the solution itself as a legitimate test
%   function. Nevertheless, the optimal error estimates requires to have enough smoothness,
%   as that in Theorem~\ref{thm:MT}
% \end{remark}
%
\subsection{Description and properties of the boundary}
\label{sec:bdr} 
$\hphantom{}$
% \comment{We did not stress in the previous paper, that this local
%   coordinates require orthogonal transformations and if the equations
%   are not invariant, this produces additional terms}
% \comment{why they should not be invariant?}
We assume that the boundary $\partial\Omega$ is of class $C^{2,1}$, that
is for each point $P\in\partial\Omega$ there are local coordinates such
that in these coordinates we have $P=0$ and $\partial\Omega$ is locally
described by a $C^{2,1}$-function, i.e.,~there exist
$R_P,\,R'_P \in (0,\infty),\,r_P\in (0,1)$ and a $C^{2,1}$-function
$a_P:B_{R_P}^{2}(0)\to B_{R'_P}^1(0)$ such that
\begin{itemize}
\item   [\rm (b1) ] $\bx\in \partial\Omega\cap (B_{R_P}^{2}(0)\times
  B_{R'_P}^1(0))\ \Longleftrightarrow \ x_3=a_P(x_1,x_2)$,
\item   [\rm (b2) ] $\Omega_{P}:=\{(x,x_{3})\fdg x=(x_1,x_2)^\top
 \in  B_{R_P}^{2}(0),\ a_P(x)<x_3<a_P(x)+R'_P\}\subset \Omega$, 
%\item   [\rm (b2) ] $\bx\in\O_P:=\O\cap (B_{R_P}^{2}(0)\times
%B_{R_P}^1(0))\ \Longleftrightarrow \   x_3>a_P(x_1,x_2)$,
\item [\rm (b3) ] $\nabla a_P(0)=\bfzero,\text{ and }\forall\,x=(x_1,x_2)^\top
  \in B_{R_P}^{2}(0)\quad |\nabla a_P(x)|<r_P$,
\end{itemize}
where $B_{r}^k(0)$ denotes the $k$-dimensional open ball with center
$0$ and radius ${r>0}$.  Note that $r_P $ can be made arbitrarily
small if we make $R_P$ small enough.  In the sequel we will also use,
for $0<\lambda<1$, the  scaled open sets $\lambda\,
\Omega_P\subset \Omega_P$, defined as follows
\begin{equation*}
%  \label{eq:scaled_omega_P}
  \lambda\, \Omega_P:=\{(x,x_{3})\fdg x=(x_1,x_2)^\top
 \in
  B_{\lambda R_P}^{2}(0),\ a_P(x)<x_3<a_P(x)+\lambda R_P'\}.
\end{equation*}
To localize near  $\partial\Omega\cap \partial\Omega_P$, for $P\in\partial\Omega$, we fix smooth
functions $\xi_{P}:\setR^{3}\to\setR$ such that %$0\leq\xi_P\leq1$ and
\begin{itemize}
\item [$\rm (\ell 1)$] $\chi_{\frac{1}{2}\Omega_P}(\bx)\leq\xi_P(\bx)\leq
  \chi_{\frac{3}{4}\Omega_P}(\bx)$,
\end{itemize}
where $\chi_{A}(\bx)$ is the indicator function of the measurable set
$A$. 
%\begin{equation*}
% \hspace*{-30mm}{\rm (\ell 1)}\quad  \xi_P=\left\{
%    \begin{aligned}
%      &1\quad\text{ for }x\in\Omega_P\cap \big(B^2_{{R_P}/2}(0)\times
%      B^1_{{R_P}/2}(0)\big), 
%    \\
%    & 0\quad\text{ for }x\in\Omega_P\backslash
%    \big(B^2_{3{R_P}/4}(0)\times B^1_{3{R_P}/4}(0)\big).
%  \end{aligned}
%\right. 
%\end{equation*}
For the remaining interior estimate we  localize by a smooth function
${0\leq\xi_{00}\leq 1}$ with $\spt \xi_{00}\subset\Omega_{00}$,
where $\Omega_{00}\subset \Omega$ is an open set such that
$\dist(\partial\Omega_{00},\,\partial\Omega)>0$.  
%We first perform calculations in $\Omega_{P}$ for all
%$P\in \partial\Omega$, especially in the study of regularity of
%tangential derivatives.
%The local estimates near the boundary are obtained in two steps. In
%the first one (see Sections~\ref{sec:tan} and~\ref{sec:tanp}) we
%estimate in $\Omega_P$ only tangential derivatives as defined
%below. In the second one we use the new obtained information and
%compute the normal derivatives from the system. 
Since the boundary $\partial\Omega $ is compact, we can use an appropriate
finite sub-covering which, together with the interior estimate, yields
the global estimate.
%  and we use the covering of
% $\partial\Omega$ made with the open sets
% $\{\O_P\}_{P\in\partial\Omega}$. 

Let us introduce the tangential derivatives near the boundary. To
simplify the notation we fix $P\in \partial\Omega$, $h\in (0,\frac{R_P}{16})$,
and simply write $\xi:=\xi_P$, $a:=a_P$. We use the standard notation
$\bx =(x',x_3)^\top$ and denote by $\be^i,i=1,2,3$ the canonical
orthonormal basis in $\setR^3$. In the following lower-case Greek
letters take values $1,\, 2$. For a function $g$ with $\spt
g\subset\spt\xi$ we define for $\alpha=1,2$ tangential translations:
\begin{equation*}
\begin{aligned}
  \trap g(x',x_3) = g_{\tau _\alpha}(x',x_3)&:=g\big (x' +
  h\,\be^\alpha,x_3+a(x'+h\,\be^\alpha)-a(x')\big )\,,
%  \\
%  \tran g(x',x_3)&:=g\big (x' - h\,\be^\alpha,x_3+a(x' -
%  h\,\be^\alpha)-a(x')\big )\,;
\end{aligned}
\end{equation*}
%%
%tangential differences
%%
%\begin{equation*}
%  \Delta^+ g:=\trap g-g,\qquad\Delta^- g:=\tran g-g\,;
%\end{equation*}
%%
tangential 
differences $\Delta^+ g:=\trap g-g$, and tangential divided
differences 
%
%\begin{equation*}
  $\difp g:= h^{-1}\Delta^+ g$.  
%\end{equation*}
%
It holds that, if $ g \in W^{1,1}(\Omega)$, then we have for $\alpha=1,2$
\begin{align} 
  \label{eq:1}
  \difp g \to \td g=\partial _{\tau_\alpha}g :=\partial_\alpha g +\partial_\alpha
  a\, \partial_3 g  \qquad \text{ as } h\to 0,
\end{align} 
almost everywhere in $\spt\xi$, (cf.~\cite[Sec.~3]{mnr2}). 
%\comment{L.: Let me know if you agree with this} 
Moreover, uniform $L^q$-bounds
for $\difp g$ imply that $\partial _\tau g $ belongs to $L^q(\spt\xi)$. More precisely, if
we define, for $0<h<R_P$
\begin{equation*}
  \Omega_{P,h}=\left\{\bx\in \Omega_P\fdg x'\in B^2_{R_P-h}(0)\right\},
\end{equation*}
and if $f\in W^{1,q}_\loc(\setR^3)$, then 
\begin{equation*}
  \int_{\Omega_{P,h}}|d^+f|^q\,d\bx\leq c\int_{\Omega_{P}}|\partial_\tau f|^q\,d\bx.
\end{equation*}
Moreover, if $d^{+}f\in L^q(\Omega_{P,h_0})$, for all $0<h_0<R_P$ and all $0<h_0$ and if
 \begin{equation*}
\exists\,c_1>0:\quad   \int_{\Omega_{P,h_0}}|d^{+}f|^q\,d\bx\leq c_1\qquad
\forall\,h_0\in(0,R_P)\text{ and } \forall\,h\in(0,h_0),
 \end{equation*}
then $\partial_{\tau}f\in L^q(\Omega_P)$ and 
\begin{equation*}
  \int_{\Omega_{P}}|\partial_{\tau}f|^q\,d\bx\leq c_1.
\end{equation*}
%Moreover, we have for all %
%%\marginpar{L:Here is $W^{1,q}(\Omega)$??, since we apply it also to
%%functions not zero at the boundary\\ M: I think that is ok}
% $1<q<\infty$, $ g \in W^{1,q}(\Omega)$ and all sufficiently small
%$h>0$, that 
%\begin{align}
%  \label{eq:2}
%  \exists\,c(a)>0:\quad \norm{\difp g }_{q,\spt\xi} \le
%  c(a)\norm{\nabla g }_q .
%\end{align}
%Conversely, if $\norm{\difp g  }_{q,\spt\xi} \le C$ for all
%sufficiently small $h>0$, then 
%\begin{align}
%  \label{eq:2a}
%  \norm {\partial _\tau  g }_{q,\spt\xi} \le C.
%\end{align}
%
%Now we formulate some auxiliary lemmas related to these objects. The
%first lemma clarifies the fact that tangential translations and
%tangential differences do not commute with partial
%derivatives. Also the explicit expressions can be used to
%quantitatively estimate the so called commutation terms, as called in
%turbulence theory~\cite{Ber-Gri-John2007}.
For simplicity we denote $\nabla a:=(\partial_1a,\partial_{2}a,
0)^\top$.
%
%and use the operations $\trap {(\cdot)}$, $\tran {(\cdot)}$,
%$\Delta^+(\cdot) $, $\Delta^+(\cdot) $, $\difp {(\cdot)}$ and $\difn
%{(\cdot)}$ also for vector-valued and tensor-valued functions,
%intended as acting component-wise.
%%
%\begin{lemma}
%  \label{lem:TD1} 
%  Let $\bv\in W^{1,1}(\Omega)$ such that $\spt \bv
%  \subset\spt\xi$. Then 
%%
%\begin{equation*}
%\begin{aligned}
%  \nabla\difpm \bv &=\difpm{\nabla \bv }+\trap{(\partial_3 \bv
%    )}\otimes\difpm{\nabla a},
%  \\
%  \bD\difpm \bv &=\difpm{\bD \bv }+\trap{(\partial_3 \bv
%    )}\otimess\difpm{\nabla a},
%  \\
%  \diver\difpm \bv &=\difpm\diver \bv +\trapm{(\partial_3 \bv
%    )}\difpm{\nabla a}
%  \\
%  \nabla \bv _{\pm\tau} &= (\nabla \bv )_{\pm\tau} + \trapm{(\partial_3 \bv
%    )}\difpm{\nabla a},
%\end{aligned}
%\end{equation*}
%where $\otimess$ is defined component-wise also for scalar and tensor-valued functions.
%\end{lemma}
%%
%The second lemma is devoted to the relation between tangential
%differences and tangential translations, provided that $h$ is small
%enough. 
%%
%\begin{lemma}
%  \label{lem:TD2}
%  Let $\spt g \subset\spt\xi$. Then
%\begin{equation*}
%  \trap{(\difn g )}=-\difp g ,\quad \tran{(\difp g )}=-\difn g , \quad 
%  \difn  g_\tau = - \difp g .
%\end{equation*}
%\end{lemma}
%%
The following variant of formula of integration by parts will be often used.
%\marginpar{formulate differently with respect to the support!!}
%\marginpar{L: what to write?\\ M: I dont know what you mean.}
\begin{lemma}
  \label{lem:TD3}
  Let $\spt g\cup\spt f\subset\spt\xi=\spt\xi_P$ and $0<h<\frac{R_P}{16}$. Then
  \begin{equation*}
    \intO f\tran g \, d\bx =\intO\trap f g\, d\bx.
  \end{equation*}
  Consequently, $\intO f\difp g \, d\bx= \intO(\difn f )g\, d\bx$.
  Moreover, if in addition $f$ and $g$ are smooth enough and at least
  one vanishes on $\partial\Omega$, then 
%\marginpar{added also    this} 
\begin{equation*} \intO f\td g \, d\bx= -\intO(\td f )g\,
    d\bx.
  \end{equation*}
\end{lemma}
\subsection{A first regularity result in space}
We start proving spatial regularity for the perturbed problem in the non-degenerate case
$\delta>0$. The estimates proved in this intermediate step are uniform with respect to
$\vep>0$ and $\delta>0$ in: a) the interior and b) in the case of tangential derivatives;
estimates depend on $\vep, \delta >0$ in the normal direction. Nevertheless, this allows
later on to use the equations pointwise to prove in a different way estimates independent
of $\epsilon, \delta >0$ even near the boundary. Thus, we can pass to the limit with
$\vep \to 0$ to treat the original problem in the non-degenerate case. Finally, the
degenerate case is treated by a suitable approximation using that the estimates are
independent of $\delta>0$.

We observe that by using a translation method, the result is
proved directly for solutions and not anymore for the Galerkin approximations.
\begin{proposition}  
  \label{prop:JMAA2017-1}
  Let the tensor field $\bS$ in~\eqref{eq:eq-e} have $(p,\delta)$-structure for some $p \in
  (1,2]$ and $\delta \in (0,\infty)$, and let $\bF$ be the associated tensor field to
  $\bS$. Let $\Omega\subset\setR^3$ be a bounded domain with $C^{2,1}$ boundary and let
  $\bu_0\in W^{1,2}_0(\Omega)$ and $\ff\in L^{p'}(I\times
  \Omega)$. Then, the unique time-regular
  solution $\bue$ of the approximate problem~\eqref{eq:eq-e} satisfies for a.e.~$t\in I$
%\comment{dependence on $a_P$?}%
  \begin{align}
    \begin{aligned}
      &\|\xi_0^2\nabla \bue(t)\|_2^2+\int\limits_0^t \int\limits_{\Omega} \epsilon \,\xi_0^2 \abs{\nabla
        ^2\bue}^2+\xi_0^2 \abs{\nabla \bF(\bD\bue)}^2\,d\bx\,ds 
      \\
      &\le c      (\|\bu_0\|_{1,2},\norm{\ff}_{L^{p'}(I\times\Omega)},\norm{\xi_0}_{2,\infty},\delta) \,,
      \\[3mm]
      &\|\xi_P^2\td\bue(t)\|_2^2+\int\limits_0^t \int\limits_{\Omega} \epsilon \,\xi^2_P
      \abs{\td \bD\bue}^2+\xi^2_P \abs{\td \bF(\bD\bue)}^2\,d\bx\,ds
      \\
      &\le c
      (\|\bu_0\|_{1,2},\norm{\ff}_{L^{p'}(I\times\Omega)},\norm{\xi_P}_{2,\infty},\norm{a_P}_{C^{2,1}},\delta)
      \,. \hspace*{-3mm}
    \end{aligned} \label{eq:est-eps}
  \end{align}
  % \marginpar{Better to call it $\nabla_\tau$???}
  Here $\xi_{0}(\bx)$ is a cut-off function with support in the interior of
  $\Omega$ and, for arbitrary $P\in \partial \Omega$, the tangential
  derivative 
  %\footnote{Here special
   % care has to be taken to distinguish the partial time-derivative $\partial_t \bue$,
   % from the tangential space-derivative $\partial_\tau\bue$}
  is defined locally in $\Omega_P$ by~\eqref{eq:1}.

Moreover, there exists a constant $C_1>0$
  such that\footnote{Recall that $c(\delta^{-1})$ indicates a possibly critical dependence
    on $\delta$ as $\delta\to 0$.} for a.e. $t\in I$
  \begin{equation}
    \label{eq:est-eps-1}
    \begin{aligned}
      &\|\xi_P^2\partial_3\bue(t)\|_2^2+\int\limits_0^t \int\limits_{\Omega} \epsilon \,\xi^2_{P}
      \abs{\partial_3 \bD\bue}^2+\xi^2_{P} \abs{\partial _3 \bF(\bD\bue)}^2\,d\bx\,ds
      \\
      & \le c
      (\|\bu_0\|_{1,2},\,\norm{\ff}_{L^{p'}(I\times \Omega)},\norm{\xi_{P}}_{2,\infty},\norm{a_{P}}_{C^{2,1}},\delta^{-1},
      \vep^{-1},C_1),
    \end{aligned}
  \end{equation}
  provided that in the local description of the boundary there holds $r_P<C_1$ in $(b3)$,
  where $\xi_{P}(\bx)$ is a cut-off function with support in $\Omega_P$.
\end{proposition}
\begin{remark}
  Proposition~\ref{prop:JMAA2017-1} and Proposition~\ref{thm:existence_perturbation} imply that
  $\bue(t) \in W^{2,2}(\Omega)$ and $\frac{\partial \bue}{\partial
    t}(t)\in L^2(\Omega)$ for a.e.~$t \in I$.
  Hence, equations~\eqref{eq:eq-e} hold pointwise a.e.~in~$I$.
\end{remark}
% \begin{remark}
%   We also point out that the estimates on $\frac{\partial
%     \bF(\bD\bue)}{\partial t}$ we prove are
%   necessary for the error estimates, but they will not be used in the study of the space
%   regularity.  Hence, with minor changes one can prove the existence of a unique
%   \underline{strong} solution to~\eqref{eq:pfluid} which is such that the equations still
%   hold almost everywhere in $I\times\Omega$ and in a weak sense with
%   \begin{equation*}
%     \bu\in L^{\infty}(I;W^{1,2}_{0}(\Omega)\cap W^{1,2}(I;L^{2}(\Omega))\quad \bF(\bD\bu)\in
%     L^{2}(I;W^{1,2}(\Omega)), 
%   \end{equation*}
%   under the less restrictive assumptions that $\bu_{0}\in W^{1,2}(\Omega)$ and $\bff\in
%   L^{p'}(I;L^{p'})$, and avoiding the assumptions leading to
%   estimate~\eqref{eq:estimate-partial-t-F}.
% \end{remark}

\begin{proof}[Proof of Proposition~\ref{prop:JMAA2017-1}]
    Fix $P\in \partial \Omega$ and use in $\Omega_P$ 
  \begin{equation*}
    \bv=\difn{(\xi^2\difp(\bue |_{\frac 12 {\Omega}_P}))},
  \end{equation*}
  where $\xi:=\xi_P$, $a:=a_P$, and $h\in(0,\frac{R_{P}}{16})$, as a
  test function in~\eqref{eq:weak-eps}. This
  yields, after integration by parts over $\Omega$, for a.e.~$t \in I$
%\comment{which weak formulation you use? maybe ds is too much}
  \begin{equation*}
%  \label{eq:dtTx2}
  \begin{aligned}
    \intO &\xi^2\difp{\frac{\partial\bue}{\partial t}}(t)\cdot \difp\bue (t)\, d\bx + \intO
    \xi^2\difp{\T^\vep(\bD\bue (t))}\cdot \difp \bD\bue  (t)\, d\bx
    \\
    =&-\intO \T^\vep(\bD\bue (t))\cdot\big(\xi ^2 \difp \partial _3 \bue (t) -(\xi_{-\tau }
    \difn\xi +\xi \difn\xi) \partial_3\bue  (t)\big) \otimess\difn\nabla a\, d\bx
    \\
    &-\intO \T^\vep(\bD\bue (t))\cdot \xi^2 \trap{(\partial _3\bue (t))}\otimess
    \difn{\difp{\nabla a}} - \T^\vep(\bD\bue (t))\cdot\difn{\big(2\xi \nabla \xi \otimess
      \difp\bue (t)\big)}\, d\bx
    \\
    &+\intO \T^\vep(\trap{(\bD\bue (t))})\cdot \big(2 \xi \partial_3\xi \difp\bue (t) + \xi^2
    \difp\partial_3\bue (t) \big)\otimess\difp\nabla a \, d\bx
    \\
    &+\intO\ff (t)\cdot\difn(\xi^2 \difp \bue (t))\, d\bx=:\sum_{j=1}^{8}\mathcal{I}_j\,.
  \end{aligned}
\end{equation*}
Hence, by using the estimates for $\mathcal{I}_j$ as
in~\cite[Proposition~3.1]{br-plasticity} (see
also~\cite[Proposition~4.4]{br-reg-shearthin}) and by observing that
$d^+\frac{\partial\bue}{\partial t}=\frac{\partial d^+\bue}{\partial t}$, one gets
\begin{align*}
  &\frac{d}{dt}\,\frac{1}{2}\intO\xi^2|\difp\bue(t)|^2\,d\bx+\vep \intO \, \xi^2 \bigabs{ \difp
    \nabla \bue(t) }^2 +\xi^2\bigabs{\nabla \difp\bue(t) }^2 d\bx
  \\
  &\quad  +\intO \xi^2 \bigabs{ \difp \bF (\bD\bue(t)) }^2 \!+\! \phi (\xi \abs{\difp\nabla
    \bue(t)}) + \phi (\xi \abs{\nabla \difp \bue(t)})\, d\bx \notag
  \\
  &\le
  c(\norm{\xi}_{2,\infty},\norm{a}_{C^{2,1}},\delta)\int\limits_\Omega|\ff(t)|^{p'}+\epsilon|\nabla\bue(t)|^2
  +\varphi(|\bD\bue(t)|)\,d\bx\,, 
\end{align*}
and, after integration over $[0,t]\subseteq I$ and the use of the a priori
estimate~\eqref{eq:main-apriori-estimate2}, we get
\begin{align*}
  &\frac{1}{2}\intO\xi^2|\difp\bue(t)|^2\,d\bx+\vep\int\limits_{0}^{t}\intO \, \xi^2 \bigabs{ \difp
    \nabla \bue(s) }^2 +\xi^2\bigabs{\nabla \difp\bue(s) }^2 d\bx\,ds
  \\
  & \quad+\int\limits_{0}^{t}\intO \xi^2 \bigabs{ \difp \bF (\bD\bue(s)) }^2 \!+\! \phi (\xi
  \abs{\difp\nabla \bue(s)}) + \phi (\xi \abs{\nabla \difp \bue(s)})\, d\bx\,ds \notag
  \\
  & \le \frac{1}{2}\intO|\xi^2\difp\bu_0|^2\,d\bx +\frac{1}{2}\|\bu_0\|_2^2+
  c(\norm{\xi}_{2,\infty},\norm{a}_{C^{2,1}},\delta)\int\limits_0^T\|\ff(t)\|_{p'}^{p'}\,dt
  \\
  & \le \frac{1}{2}\|\bu_0\|^2_{1,2}+
  c(\norm{\xi}_{2,\infty},\norm{a}_{C^{2,1}},\delta)\int\limits_0^T\|\ff(t)\|_{p'}^{p'}\,dt\,,
\end{align*}
from which~$(\ref{eq:est-eps})_2$ follows by standard arguments, using
that the estimates are independent of $h>0$.

\smallskip

The same argument used with a test function $\xi_{00}$ with compact support in $\Omega$,
and standard finite differences can be used to prove~$(\ref{eq:est-eps-1})_1$: this
implies that
\begin{equation*}
  \bue\in L^\infty(I;W^{1,2}_{loc}(\Omega))\cap
  L^2(I;W^{2,2}_{loc}(\Omega))\qquad\text{and}\qquad
  \bF(\bD\bue)\in L^2(I;W^{1,2}_{loc}(\Omega)).
\end{equation*}
Coupling the latter with the time regularity from Proposition~\ref{thm:existence_perturbation}
one obtains that the equations~\eqref{eq:eq-e}  hold pointwise
a.e.~in $I\times \Omega$.

\bigskip

We prove now the result on the regularity in the normal direction
from~\eqref{eq:est-eps-1}. We re-write the equations in~\eqref{eq:eq-e} as follows
  \begin{equation*}
    \label{eq:linear_system}
    -\frac {\partial u^i_\vep}{\partial t}   + \sum_{k=1}^3\partial_{k 3}S_{i 3}^\vep(\bD\bue)\partial_3
    D_{k 3}\bue +\partial_{3\alpha}S_{i 3}^\vep(\bD\bue)\partial_3
    D_{ 3\alpha}\bue  =\mathfrak{f}^i\qquad\textrm
    {a.e.\ in }I\times \Omega\,,
  \end{equation*}
  where $\mathfrak {f}^i :=-f^i -\partial_{\gamma \sigma}S_{i 3}^{\vep}(\bD\bue)\partial_3
  D_{\gamma \sigma}\bue- \sum_{k,l=1}^3\partial_{k l}S_{i \beta}^{\vep}(\bD\bue)\partial_\beta
  D_{k l}\bue$, $i=1,2,3$.
  We now proceed as in~\cite[Eq.~(3.3)]{br-plasticity} and we multiply these equations by
  $\partial _3 \widehat D_{i 3}\bue$, where $\widehat D_{\alpha\beta}\bue =0$, for
  $\alpha,\beta=1,2$, $\widehat D_{\alpha 3}\bue =\widehat D_{3\alpha}\bue
  =2D_{\alpha3}\bue$, for $\alpha=1,2$, $\widehat D_{33}\bue =D_{33}\bue $ and sum over
  $i=1,2,3$. In this way we get a lower bound on the nonlinear term from the
  left-hand-side  in such a way that
  \begin{equation*}
    -\sum_{i=1}^3\frac {\partial u^i_\vep}{\partial t}\partial_3\widehat{D}_{i3}\bue+    \left(
      \epsilon+\phi''(|\bD\bu_\vep|)\right)| {\boldsymbol { \mathfrak b}}|^2 \leq
    c    |\boldsymbol { \mathfrak f}|| {\boldsymbol {   \mathfrak b}}|\qquad\textrm {a.e.\ in
    }I\times \Omega\,,
  \end{equation*}
where $\mathfrak b_i :=\partial_3 D_{i3}\bue$.

  By straightforward manipulations (cf.~\cite[Sections 3.2~and~4.2]{br-plasticity}) we
  have  that for a.e. $(t,\bx)$ in $I\times \Omega_P$ 
  \begin{equation*}
    \begin{aligned}
      & |\boldsymbol { \mathfrak f}| \leq c \left(|\bff|
        +(\epsilon+\phi''(|\bD\bu_\vep|))\left(|\partial_\tau\nabla \bu_\vep|+\|\nabla
          a\|_{\infty}|\nabla^2 \bu_\vep|\right)\right)
      \\
      & |\boldsymbol { \mathfrak b}|\geq2|\widetilde{\boldsymbol{\mathfrak
          b}}|-|\partial_{\tau}\nabla\bu_\vep|-\|\nabla a\|_{\infty}|\nabla ^{2}\bu_\vep|,
    \end{aligned}
  \end{equation*}
  for $\widetilde {\mathfrak b}_i:=\partial ^2_{33}u^i_\vep$, $i=1,2,3$. These results
  imply that a.e.~in $I\times \Omega_P$
  \begin{align*}
    -&\sum_{l=1}^3\frac {\partial u^l_\vep}{\partial t}\partial_3\widehat{D}_{l3}\bue+ \left(
      \epsilon+\phi''(|\bD\bu_\vep|)\right)| {\widetilde {\boldsymbol { \mathfrak b}}}|^2
    \\
    & \leq c \left[|\bff|+(\epsilon+\phi''(|\bD\bu_\vep|))\left(|\partial_\tau\nabla
        \bu_\vep|+\|\nabla a\|_{\infty}|\nabla^2 \bu_\vep|\right)\right]  |{{\boldsymbol {
        \mathfrak b}}}|.
  \end{align*}
  We then add on both sides, for
  $\alpha=1,2$ and $i,k=1,2,3$ the term
  \begin{equation*}
    \left(\vep+\phi''(|\bD\bu_\vep|)\right)\,|\partial_\alpha\partial_i
    u^k_\vep|^2\,,
  \end{equation*}
  and estimate $\boldsymbol{\mathfrak{b}}$ with all second order spatial derivatives
  obtaining
  \begin{equation*}
    \begin{aligned}
      -&\sum_{l=1}^3\frac{\partial u^l_\vep}{\partial t}\partial_3\widehat{D}_{l3}\bue+ \left(
        \epsilon+\phi''(|\bD\bu_\vep|)\right)| {\nabla ^2\bu_\vep }|^2
      \\
      &\leq \frac{\epsilon}{4}| {\nabla ^2\bu_\vep }|^2+
      \frac{1}{4}(\epsilon+\phi''(|\bD\bu_\vep|))|\nabla^2 \bue|^2
      \\
      &\quad +c \left(|\bff|^2 +(\epsilon+\phi''(|\bD\bu_\vep|))\left(|\partial_\tau\nabla
          \bu_\vep|^2+\|\nabla a\|_{\infty}|\nabla^2 \bu_\vep|\right)\right)\,,
    \end{aligned}
  \end{equation*}
  where in the right-hand side we used also the definition of the tangential derivative
  (cf.~\eqref{eq:1}).  Next, we choose the open sets $\Omega_{P}$ in such a way that
  $\|\nabla a\|_{\infty}=\|\nabla a_{P}(x_1,x_2)\|_{{\infty},\Omega_{P}}$ is small enough,
  so that we can absorb the last term from the right-hand side. We finally arrive at the
  following pointwise inequality
  \begin{equation}
   \label{eq:pointwise-estimate}
   \begin{aligned}
      -&\sum_{i=1}^3\frac {\partial u^i_\vep}{\partial t}\partial_3\widehat{D}_{i3}\bue+ \left(
        \epsilon+\phi''(|\bD\bu_\vep|)\right)| {\nabla ^2\bu_\vep }|^2
      \\
      &
      \leq c \left(|\bff|^2+(\epsilon+\phi''(|\bD\bu_\vep|))\left(|\partial_\tau\nabla
          \bu_\vep|^2\right)\right)\quad \text{a.e. in }I\times \Omega_P\,.
    \end{aligned}
  \end{equation}
  We neglect $\phi''(|\bD\bu_\vep|)$ (which is non-negative) from the left-hand side, 
  multiply by $\xi^2$, and integrate in the spatial variable over the whole domain
  $\Omega$. In particular, since $\bue$ and $\frac{\partial
    \bue}{\partial t}$ both vanish on
  $I\times\partial\Omega$, the first term coming from the left-hand side
  of~\eqref{eq:pointwise-estimate} can be written as follows by performing some
  integration by parts:
  \begin{equation*}
    \begin{aligned}
      &      \intO      -\sum_{i=1}^3\xi^2 \frac{\partial
      u^i_\vep}{\partial t}\partial_3\widehat{D}_{i3}\bue\,d\bx 
      \\
      &    =\intO
      \xi^2\frac{1}{2}\frac{d}{dt}|\partial_3\bue|^2-\sum_{\alpha=1}^2\xi^2\frac
      {\partial
      u_\vep^\alpha}{\partial t} \,\partial^2_{3\alpha}u_\vep^3+\sum_{i=1}^3\xi\,\partial_3 \xi\frac{\partial
      u_\vep^i}{\partial t}\partial_3 u_\vep^i\,d\bx.
    \end{aligned}
  \end{equation*}
  After a further integration over $[0,t]\subseteq I$ we obtain
  from~\eqref{eq:pointwise-estimate}, the
  tangential regularity already proved in $(\ref{eq:est-eps})_2$ and
  Korn's inequality the following inequality
\begin{equation*}  
  \begin{aligned}
    &\|\xi^2\partial_3\bue(t)\|_2^2+\vep\int\limits_{0}^{t}\intO|\nabla
    ^2\bue|^2\,d\bx\,ds
    \\
    &\leq \|\xi^2\partial_3\bu_0\|_2^2 + c(\vep^{-1},
    \norm{\xi}_{2,\infty},\norm{a}_{C^{2,1}},\delta^{-1})\int\limits_0^T\|\ff(s)\|_{p'}^{p'}+\Bignorm{\frac{\partial\bue}{\partial
      t}(s)}_2^2
    ds,
  \end{aligned}
\end{equation*}
hence the boundedness of the right-hand side,  by using
Proposition~\ref{thm:existence_perturbation}.  

With this estimate and recalling the properties of the covering we
finish the proof.
\end{proof}
\subsection{Uniform estimates for the second order spatial derivatives}
We now improve the estimate in the normal direction in the sense that we will show that
they are bounded uniformly with respect to $\vep, \delta >0$. The used method is an
adaption to the time evolution problem of the treatment in~\cite{br-plasticity}, which is
based on previous results from~\cite{SS00}.
\begin{proposition}
  \label{prop:main}
  Let the same hypotheses as in Theorem~\ref{thm:MT} be satisfied with $\delta >0$ and let
  the local description $a_P$ of the boundary and the localization function $\xi_P$
  satisfy $(b1)$--\,$(b3)$ and $(\ell 1)$ (cf.~Section~\ref{sec:bdr}). Then, there exists
  a constant $C_2>0$ such that the time-regular solution $\bue\in L^{\infty}(I;W^{1,2}_0(\Omega))\cap
  L ^2(I;W^{2,2}(\Omega))$ of
  the approximate problem~\eqref{eq:eq-e} satisfies\footnote{Recall that $c(\delta)$ only
    indicates that the constant $c$ depends on $\delta$ and will satisfy $c(\delta)\le
    c(\delta_0)$ for all $\delta\le \delta_0$.  }  for every $P\in
  \partial \Omega$ and for a.e.~$t\in I$
  \begin{equation}\label{eq:main-est}
      \int\limits_\Omega\xi_P^2|\partial_3\bue(t)|^2\,d\bx+\int\limits_0^t \int\limits_\Omega \epsilon
      \,\xi^2_P |\partial_3\bD\bue|^2 + \xi^2_P |\partial_3\bF(\bD\bue)|^2\,d\bx\,ds
     \leq c\,, 
  \end{equation}
  provided $r_P<C_2$ in $(b3)$, with $c$ depending on
  $ \|\bu_0\|_{2,2},\norm{\divo \bS(\bD\bu_0)}_2,
  \norm{\ff}_{L^{p'}(I \times\Omega)}$, $
  \norm{\ff}_{L^{2}(I\times \Omega)},
  \norm{\frac{\partial\ff}{\partial
      t}}_{L^{2}(I\times \Omega)},\norm{\xi_P}_{2,\infty},\norm{a_P}_{C^{2,1}},\delta,
  C_2$.
\end{proposition}
\begin{proof}
  We adapt the strategy as in~\cite[Proposition~3.2]{br-plasticity} to the time-dependent
  problem.  Fix an arbitrary point $P\in \partial \Omega$ and a local description $a=a_P$
  of the boundary and the localization function $\xi=\xi_P$ satisfying $(b1)$--\,$(b3)$
  and $(\ell 1)$. In the following we denote by $C$ constants that depend only on the
  characteristics of $\bS$.  First we observe that, by the results of
  Proposition~\ref{prop:equivalence} there exists a constant $C_0$, depending only on the
  characteristics of $\bS$, such that\footnote{In this section we do not write explicitly
    the dependence on space and time variables, since the reader at this point will be
    acquainted enough with the matter to avoid heavy notation.}
  \begin{equation*}
   \frac{1}{C_0}| \partial_3\bF(\bD\bue)|^2\leq
   \mathbb{P}_3(\bD\bue)\qquad \text{a.e.  in }I\times\Omega.
  \end{equation*}
  Thus, we get, using also the symmetry of both $\bD\bue$ and $\bS$,
  \begin{equation*}
    \begin{aligned}
      &\int\limits_\Omega \vep \,\xi^2|\partial_3 \bD\bue|^2 + \frac{1}{C_0} \xi^2
      |\partial_3\bF(\bD\bue)|^2\,d\bx
      \\
      &\leq\int\limits_\Omega\xi^2 \big (\vep \,\partial_3 D_{\alpha\beta}\bue +\partial_3
      S_{\alpha\beta} (\bD\bue) \big) \partial_3 D_{\alpha\beta}\bue\,d\bx
      \\
      &\quad + \int\limits_\Omega\xi^2\big(\vep \,\partial_3D_{3\alpha}\bue+ \partial_3S_{3\alpha
      }(\bD\bue)\big)\partial _\alpha D_{33}\bue\,d\bx
      \\
      &\quad +\int\limits_\Omega \sum_{j=1}^3 \xi^2\partial_3\big ( \vep\, D_{j3}\bue
      +S_{j3}(\bD\bue)\big ) \partial_3^2 u^j_\vep\,d\bx
      % \\
      % &\quad 
      =:\mathcal{J}_{1}+\mathcal{J}_{2}+\mathcal{J}_{3}\,.
    \end{aligned}
  \end{equation*}
%\comment{for consistency of notation we should indicate a $t$ here and
%  in the following, but
%  I don't know if we really want that. L. I do not want, we can simply write that we forget
%time dependence when not needed}
  The terms $\mathcal{J}_{1}$ and $\mathcal{J}_{2}$ can be estimated exactly as
  in~\cite{br-plasticity} to prove, for $\lambda>0$, that
\begin{equation*}
  \begin{aligned}
    |\mathcal{J}_{1}|+|\mathcal{J}_{2}| \leq &\,\param\int\limits_\Omega\xi^2
    |\partial_3\bF(\bD\bue)|^2 + %\,d\bx+\param \int\limits_\Omega
    \vep\,\xi^2|\partial_3\bD\bue|^2\,d\bx
    \\
    &+c_{\param^{-1}}(1+\|\nabla a\|_\infty^2)\sum_{\beta=1}^2\int\limits_\Omega\xi^2
    |\partial_\beta\bF(\bD\bue)|^2 +\vep\,\xi^2|\partial_\beta\bD\bue|^2 \,d\bx
    \\
    &+c_{\param^{-1}}\sum_{\beta=1}^2\int\limits_\Omega\xi^2
    |\partial_{\tau_\beta}\bF(\bD\bue)|^2+ \vep\,\xi^2|\partial_{\tau_\beta}\bD\bue|^2\,d\bx
    \\
    % & + c_{\param^{-1}}(1+\|\nabla a\|_\infty^2) \sum_{\beta=1}^2\int\limits_\Omega
    % \vep\,\xi^2|\partial_\beta\bD\bue|^2\,d\bx
    % \\
    &+c_{\param^{-1}} \big(1+\|\nabla\xi\|_\infty^2+\|\nabla a\|^2_\infty\big
    )\int_\Omega\phi(|\bD\bue|)+\vep\,|\bD\bue|^2\, d\bx,
  \end{aligned}
\end{equation*}
for some constant $c_{\param^{-1}} $ depending only on $\param^{-1}$.
The term $\mathcal{J}_3$ can be estimated by observing that we can re-write the
equations~\eqref{eq:eq-e} as follows 
  \begin{equation*}
    \partial_3\big(\epsilon\, D_{j3}\bue+
    S_{j3}(\bD\bue)\big)=\frac {\partial u_\vep^j}{\partial t}-f^j-\partial_\beta\big (\epsilon\,D_{j\beta}\bue+
    S_{j\beta}(\bD\bue)\big)\qquad \text{a.e. in     }I\times\Omega\,.
  \end{equation*}
  Hence, we can multiply by $\bue$ and integrate by parts in space, since
  $\bue=\frac{\partial \bue}{\partial t} =\mathbf{0}$ on
  $I\times\partial\Omega$.  We treat the terms 
  without time derivative as $I_3$ in~\cite[p.~186]{br-plasticity} and integrate by parts
  the one involving $\partial_t\bue$ to get the following 
\begin{equation*}
  \begin{aligned}
    &\mathcal{J}_{3}=
    \\
    &=\sum_{j=1}^3\intO \xi^2\frac{\partial u_\vep^j}{\partial t}\,\partial^2_{33}u_\vep^j-\xi^2\big(f^j
    +\partial_\beta S_{j\beta}(\bD\bue)+\epsilon\partial_\beta D_{j\beta}\bue\big)\big
    (2\partial _3D_{j3}\bue - \partial _j D_{33}\bue\big )\,d\bx\,
    \\
    &=-\frac{1}{2}\frac {d}{dt}\intO \xi^2  |\partial_3\bue|^2\,d\bx
    -2\sum_{j=1}^3 \intO \xi \partial _3\xi \frac{\partial u_\vep
      ^j}{\partial t} \partial _3u^j_\vep \, d\bx
    \\
    &\qquad -\sum _{j=1}^3\intO\xi^2\big(f^j
    +\partial_\beta S_{j\beta}(\bD\bue)+\epsilon\partial_\beta D_{j\beta}\bue\big)\big
    (2\partial _3D_{j3}\bue - \partial _j D_{33}\bue\big )\,d\bx\,
    \\
    &\leq -\frac{1}{2} \frac d{dt}\intO \xi^2|\partial_3\bue|^2\,d\bx+
    \param\, C \!\int\limits_\Omega\xi^2|\partial_3\bF(\bD\bue)|^2\,d\bx
    +c_{\param^{-1}}\sum_{\beta=1}^2\int\limits_\Omega\xi^2|\partial_\beta\bF(\bD\bue)|^2\,d\bx
    \\
    &\quad+\param \! \int\limits_\Omega \vep \,\xi^2|\partial_3\bD\bue|^2\,d\bx +c_{\param^{-1}}
    \sum_{\beta=1}^2\int\limits_\Omega \vep
    \,\xi^2|\partial_\beta\bD\bue|^2\,d\bx +c\int\limits_\Omega \xi^2|\partial_3 \bue|^2\,d\bx
    \\
    &\quad  + c\,\norm{\nabla\xi }_\infty ^2 \Bignorm{\frac{\partial
        \bue }{\partial t}}_2^2+c_{\param^{-1}}\big(\|\bff\|_{p'}^{p'}+\|\bD\bue\|_p^p +\delta^p\,\big)\,.
\end{aligned}
\end{equation*}
In these estimates we use for the terms with
$\partial _\beta \bF(\bD\bue)$ and $\partial _\beta \bD\bue$ the
definition of the tangential derivative in \eqref{eq:1} to get
\begin{align*}
  &\int\limits_\Omega\xi^2
  |\partial_\beta\bF(\bD\bue)|^2 +\vep\,\xi^2|\partial_\beta\bD\bue|^2
    \,d\bx
  \\
  &\le \int\limits_\Omega\xi^2
  |\partial_{\tau_\beta}\bF(\bD\bue)|^2 +\vep\,\xi^2|\partial_{\tau_\beta}\bD\bue|^2
    \,d\bx
  \\
  &\quad + \norm{\nabla a}^2_\infty\int\limits_\Omega\xi^2
  |\partial_3\bF(\bD\bue)|^2 +\vep\,\xi^2|\partial_3\bD\bue|^2
    \,d\bx\,.
\end{align*}
Note that such terms already are present in the estimates 
for $\{\mathcal{J}_{i}\}_{i=1,2,3}$. Now we choose the covering
such that $\|\nabla a\|_\infty$ is small enough and only at this point we fix $\lambda>0$
small enough (in order to absorb in the left-hand side terms involving $\partial_3\bD\bue$
and $\partial_3\bF(\bD\bue)$). We then obtain after integration in time over
$[0,t]\subseteq[0,T]$ the following estimate
\begin{equation*}
  \begin{aligned}
    &    \intO\xi^3|\partial_3\bue(t)|^2\,d\bx+ \int\limits_0^t\int\limits_\Omega \vep
    \,\xi^2|\partial_3\bD\bue|^2+ \frac 1{C_0} \xi^2|\partial_3\bF(\bD\bue)|^2\,d\bx\,ds
    \\
    & \leq     \intO\xi^3|\partial_3\bu_0|^2\,d\bx +c %_{\param^{-1}}%(1+\|\nabla a\|_{\infty}^2)
    \sum_{\beta=1}^2 \int\limits_0^T \int\limits_\Omega\xi^2|\partial_{\tau_{\beta}}\bF(\bue)|^{2}+\vep
    \,\xi^2|\partial_{\tau _\beta}\bD\bue|^2\,d\bx\,ds
    \\
    &\quad+c %_{\param^{-1}}\big(1+ \norm{\nabla ^2
    % a}^2_\infty)(1+\|\nabla\xi\|_{\infty}^2\big)
    \int\limits_0^T\!\intO
    |\bff|^{p'} \! +\phi(|\bD\bue|)+\phi(\delta) +\vep \abs{\bD\bue}^2
    +\Bigabs{\frac{\partial \bue}{\partial t}}^2d\bx\,ds + \int\limits_0^t\!\intO\xi^3|\partial_3\bue|^2\,d\bx\,ds.
    % \\
    % &\quad + c_{\param^{-1}}\big(1+ \norm{\nabla ^2
    %   a}^2_\infty)(1+\|\nabla\xi\|_{\infty}^2\big)\,\int\limits_0^t\|\bD \bue(s) \|_{2}^{2}\,ds.
    \end{aligned}
\end{equation*}
Using the uniform estimates \eqref{eq:main-apriori-estimate2},
\eqref{eq:estimate-partial-t-F}  and \eqref{eq:est-eps} we can
apply Gronwall's inequality to prove the estimate \eqref{eq:main-est}.
% The right-hand side is then bounded in terms of the data from the result of the
% Proposition~\ref{prop:JMAA2017-1}.
\end{proof}
Choosing now an appropriate finite covering of the boundary (for the
details see also~\cite{br-reg-shearthin}),
Propositions~\ref{prop:JMAA2017-1}, \ref{prop:main} yield the
following result:
\begin{proposition}
\label{thm:estimate_for_ue}
Let the same hypotheses as in Theorem~\ref{thm:MT} with $\delta>0$ be satisfied. Then, it
holds\footnote{Recall that $c(\delta)$ only indicates that the constant $c$ depends on
  $\delta$ and will satisfy $c(\delta)\le c(\delta_0)$ for all $\delta\le \delta_0$.  }
for all $t\in I$
  \begin{equation*}
    \|\nabla\bue(t)\|^2_{2}+    \int\limits_0^t\epsilon\,\|\nabla\bD\bue(s)\|^2_2 + \| \nabla
    \bF(\bD\bue(s))\|^2_2\,ds\leq 
    C %(\|\nabla\bu_0\|_2, \norm{\ff}_{L^{p'}(I\times \Omega)},\delta, \partial \Omega).
  \end{equation*}
  with $C$ depending on   $ \|\bu_0\|_{2,2},\norm{\divo \bS(\bD\bu_0)}_2,
  \norm{\ff}_{L^{p'}(I \times\Omega)}$, $
  \norm{\ff}_{L^{2}(I\times \Omega)},
  \norm{\frac{\partial\ff}{\partial
      t}}_{L^{2}(I\times \Omega)}$,$\delta$,  $\partial \Omega$ and
  the characteristics of $\bS$.
  %$\norm{\xi_P}_{2,\infty},\norm{a_P}_{C^{2,1}}, \delta,  C_2$.
\end{proposition}
%\begin{remark}
%  Observe that in order to obtain the estimate on the time derivative one can use
%  $\bu_0\in W^{2,2}(\Omega)$ to justify $\divo \bS(\bD\bu_0)\in L^2(\Omega)$, being
%  $\delta>0$. In the case $\delta=0$ one has to change it with the dependence on
%  $\|\bu_0\|_{W^{1,2}}+\|\divo \bS(\bD\bu_0)\|$.
%\end{remark}
\subsection{Passage to the limit}
Since the estimates in Propositions~\ref{thm:existence_perturbation},
\ref{thm:estimate_for_ue} are uniform with respect to $\vep>0$, they are inherited by $\bu=\lim_{\epsilon\to0}\bue$. The function $\bu$ is the unique
\underline{regular} solution to the initial boundary value problem~\eqref{eq:pfluid}.
We can now prove the existence result for regular solutions.
\begin{proof}[Proof (of Theorem~\ref{thm:MT})]
  First, let us assume that $\delta >0$. From
  Proposition~\ref{thm:existence_perturbation}, Proposition~\ref{lem:hammer}, and
  Proposition~\ref{thm:estimate_for_ue} we know that $\bF(\bD\bue)$ is uniformly bounded with
  respect to $\vep$ in $W^{1,2}(I\times\Omega)$. This also implies (cf.~\cite[Lemma
  4.4]{bdr-7-5}) that $\bue$ is uniformly bounded with respect to $\vep $ in
  $L^p(I;W^{2,p}(\Omega))\cap W^{1,p}(I;W^{1,p}(\Omega))$. The properties of $\bS$ and
  Proposition~\ref{thm:existence_perturbation} also yield that $\bS(\bD\bue)$ is uniformly
  bounded with respect to $\vep $ in $L^{p'}(I\times\Omega)$.  Thus, there exists a
  sub-sequence $\{\epsilon_n\}$ (which converges to $0$ as $n\to+\infty)$, $\bu \in
  L^p(I;W^{2,p}(\Omega))\cap W^{1,p}(I;W^{1,p}_{0}(\Omega)) \cap W^{1,\infty}(I;L^2(\Omega))$, $\bF ^*\in
  W^{1,2}(I\times \Omega)$, and $\bS^* \in L^{p'}(I\times \Omega)$ such that
  \begin{equation*}
    \begin{aligned}
      \buen&\rightharpoonup \bu&&\text{in }L^p(I;W^{2,p}(\Omega)\cap
      W^{1,p}_0(\Omega))\cap W^{1,p}(I;W^{1,p}_0(\Omega)),
      \\
      \buen&\overset{*}{\rightharpoonup} \bu&&\text{in }W^{1,\infty}(I;L^2(\Omega))\,,
      \\
      \bD\buen&\to\bD\bu\quad&&\text{a.e. in }I\times\Omega\,,
      \\
      \bF(\bD\buen) &\rightharpoonup \bF^*&&\text{in
      }W^{1,2}(I\times \Omega)\,,
      \\
      \bS(\bD\buen) &\rightharpoonup \bS^*&&\text{in
      }L^{p'}(I\times \Omega)\,.
    \end{aligned}
  \end{equation*}
  The continuity of $\bS$ and $\bF$ and the classical result stating that
  the weak limit and the a.e.~limit in Lebesgue spaces coincide
  (cf.~\cite{GGZ}) implies that
  \begin{equation*}
    \begin{aligned}
      \bF^* =\bF(\bD\bu) \qquad \text{and}\qquad
      \bS^* = \bS(\bD\bu)\,.
    \end{aligned}
  \end{equation*}
  These results enable us to pass to the limit in the weak formulation
  \eqref{eq:weak-eps} of the perturbed problem~\eqref{eq:eq-e},
  which yields for all $\psi \in C_0^\infty (I)$ and all
  $\bv\in V$%W^{1,p}_0(\Omega)\cap L^2(\Omega)$
\begin{align}\label{eq:weak-2}
  \int\limits _0^T\Bighskp{\frac {\partial\bfu}{\partial t}(t)
  }{\bfv}\psi (t) \, dt +\int\limits _0^T\hskp{\bS(\bD\bfu(t))}{\bD\bfv}\psi (t)\,dt =\int\limits
  _0^T\hskp{\bff(t)}{\bfv}\, \psi(t)\, dt\,,
\end{align}
%and $\bu (0)=\bu_0$ in $W^{1,2}_0(\Omega)$, since
%  \begin{equation*}
%    \int\limits_\Omega \bS(\bD\bu)\cdot\bD\bv\,d\bx
%    =\int\limits_\Omega\bff\cdot\bv\,d\bx\qquad\forall\,\bv \in
%    C^{\infty}_0(\Omega)\,,
%  \end{equation*}
since  $\lim _{\vep_n \to 0}\int\limits_0^T \int\limits_\Omega \vep_n
  \bD\buen(t) \cdot \bD\bv\, \psi(t)\,d\bx \,dt=0$. The weak lower semi-continuity of the norm implies that
\begin{align*}
  \| \bF(\bD\bu)\|_{W^{1,2}(I\times\Omega)}
  &\leq\liminf_{\epsilon_n\to0}
    \| \bF(\bD\buen)\|_{W^{1,2}(I\times\Omega)}\,,
  \\
    \| \bu\|_{W^{1,\infty}(I;L^2(\Omega))}
  &\leq\liminf_{\epsilon_n\to0}
    \| \buen\|_{W^{1,\infty}(I;L^2(\Omega))}\,.
\end{align*}
By density and the strict monotonicity of $\bS$ we thus know that $\bu$
  is the unique regular solution of problem~\eqref{eq:pfluid}.
This proves Theorem~\ref{thm:MT} in the case $\delta>0$, since the
weak formulation~\eqref{eq:cont-var} follows immediately from~\eqref{eq:weak-2}.

\smallskip

Let us consider now the case $\delta=0$. Proposition~\ref{prop:JMAA2017-1} and
Proposition~\ref{prop:main} are valid only for $\delta>0$ and thus cannot be used directly
for the case that $\bS$ has $(p,\delta)$-structure with $\delta=0$. However, it is proved
in~\cite[Section 3.1]{bdr-7-5} that for any stress tensor with $(p,0)$-structure $\bS$,
there exist\footnote{The special case $\bS^\kappa(\bD) =\abs{\bD}^{p-2}\bD$ could be
  approximated by ${\bS^{\kappa}(\bD) :=(\kappa+\abs{\bD})^{p-2}\bD}$ as
  $\kappa\to0$. However, for a general extra stress tensor $\bS$ having only
  $(p,\delta)$-structure this is not possible.} a stress tensors $\bS^\kappa$, having
$(p,\kappa)$-structure with $\kappa>0$ approximating $\bS$ in an appropriate
way. Thus we approximate~\eqref{eq:pfluid} by the system
\begin{equation*}
%  \label{eq-ek}
  \begin{aligned}
    \frac{\partial \bu_{\vep ,\kappa}}{\partial t} -\divo \bfS^{\vep,\kappa}
    (\bfD\bu_{\vep ,\kappa})&=\bff\qquad&&\text{in }I\times\Omega\,,
    \\
    \bfu_{\vep,\kappa} &= \bfzero &&\text{on } I\times\partial \Omega\,,
    \\
    \bfu_{\vep,\kappa}(0) &= \bu_0 &&\text{in } \Omega\,,
  \end{aligned}
\end{equation*}
where 
  \begin{equation*}
%    \label{eq:perturbed_Sk}
    \bS^{\epsilon,\kappa}(\bQ):=\epsilon \,\bQ +
    \bS^\kappa(\bQ),\qquad\text{with }\epsilon>0\,,\, \kappa\in (0,1)\,.
  \end{equation*}
  For fixed $\kappa>0$ we can use the above theory and use that fact that the estimates
  are uniform in $\vep$ to pass to the limit as $\epsilon\to0$. Thus, we obtain that
  for all $\kappa \in (0,1)$ there exists a unique $\bu_\kappa \in
  L^p(I;W^{1,p}_0(\Omega))$
  fulfilling
  \begin{align*}
    \norm{ \bfu_\kappa}_{W^{1,\infty}(I; L^2(\Omega))} +
    &\norm{\bF(\bD\bfu_\kappa)}_{W^{1,2}(I\times\Omega)}
      \leq c_0   ({\ff}, \bu_0,\partial \Omega) \,, 
  \end{align*}
  satisfying  for all $\psi \in C_0^\infty (I)$ and all
  $\bv\in W^{1,2}_0(\Omega)$
\begin{align*}%\label{eq:weak-2}
  \int\limits _0^T\Bighskp{\frac {\partial\bfu_\kappa}{\partial t}(t)
  }{\bfv}\psi (t) \, dt +\int\limits _0^T\hskp{\bS^\kappa(\bD\bfu_\kappa(t))}{\bD\bfv}\psi (t)\,dt =\int\limits
  _0^T\hskp{\bff(t)}{\bfv}\, \psi(t)\, dt\,,
\end{align*}
%and $\bu_\kappa (0)=\bu_0$ in $W^{1,2}_0(\Omega)$.
The constant $c_0$ is independent of $\kappa\in(0,1)$ and $\bF^\kappa\colon\setR^{3 \times
  3} \to \setR^{3 \times 3}_\sym$ is defined through
\begin{equation*}
%  \label{eq:def_Fk}
  \bF^\kappa(\bP):= \big (\kappa+\abs{\bP^\sym} \big )^{\frac
    {p-2}{2}}{\bP^\sym } \,.
\end{equation*}
Now we can proceed as in~\cite{bdr-7-5}. Indeed, it follows that
$\bF^\kappa(\bD\bu_\kappa)$ is uniformly bounded in $W^{1,2}(I\times\Omega)$, that
$\bu_\kappa$ is uniformly bounded in $W^{1,p}(I\times \Omega)$ and that
$\bS^\kappa(\bD\bu_\kappa)$ is uniformly bounded in $L^{p'}(I\times\Omega)$. Thus, there
exist $\bF^* \in W^{1,2}(I\times\Omega)$, $\bu\in L^{p}(I;W^{1,p}_0(\Omega))$, $\bS^*
 \in L^{p'}(I\times \Omega)$, and a sub-sequence $\{\kappa_n\}$, with $\kappa_n\to0$,
such that
  \begin{equation*}
    \begin{aligned}
      \bF^{\kappa_n}(\bD\bu_{\kappa_n}) &\rightharpoonup \bF^*&&\text{in
      }W^{1,2}(I\times\Omega)\,,
      \\
      \bF^{\kappa_n}(\bD\bu_{\kappa_n})&\to \bF^*\quad&&\text{in
      }L^{2}(I\times\Omega) \text{ and a.e.~in }I\times \Omega\,,
      \\
      \bu_{\kappa_n}&\rightharpoonup \bu&&\text{in }
      L^p(I;W^{1,p}_0(\Omega))\,,
      \\
      \bS^\kappa(\bD\bu_\kappa)&\rightharpoonup \bS^* &&\text{in } L^{p'}(I\times\Omega)\,.
    \end{aligned}
  \end{equation*}
Setting $\bB:=( \bF^0)^{-1}(\bF^*)$, it follows from
\cite[Lemma 3.23]{bdr-7-5}  that
\begin{equation*}
  \bD\bu_{\letter_n}=(\bF^{\letter_n})^{-1}(\bF^{\letter_n}(\bD\bu_{\letter_n})
  )\to( \bF^0)^{-1}(\bF^*)=\bB\quad\text{ a.e.~in }I\times  \Omega.  
\end{equation*}
Since weak and a.e.~limit coincide we obtain that
\begin{equation*}%\label{ae}
  \bD\bu_{\letter_n} \to \bD\bu=\bB \qquad \text{ a.e.~in }  I\times \Omega\,. 
\end{equation*}
From~\cite[Lemma 3.16]{bdr-7-5} and~\cite[Corollary 3.22]{bdr-7-5} it now follows that
\begin{align*}
  \begin{aligned}
    \bF^{\letter _n}(\bD\bu_{\kappa_n}) &\rightharpoonup \bF^0(\bD\bu)&&\text{in
    }W^{1,2}(I\times \Omega)\,,
    % \\
    % \bF^{\kappa_n}(\bD\bu_{\kappa_n})&\to \bF^0(\bD\bu)\quad
    % &&\text{in
    % }L^{2}(\Omega) \text{ and a.e.~in }\Omega\,,
    \\
    \bS^{\letter_n}(\bD\bu_{\letter_n}) &\to \bS(\bD\bu)
    &&\text{a.e.~in } I\times \Omega\,.
  \end{aligned}
\end{align*}
Since weak and a.e.~limit coincide we obtain that
\begin{equation*}
%\label{ae}
      \bF^* =\bF^0(\bD\bu) \qquad \text{and}\qquad
      \bS^* = \bS(\bD\bu) \qquad \text{ a.e.~in } I\times  \Omega\,. 
\end{equation*}
Now we can finish the proof in the same way as in the case
$\delta>0$. 
\end{proof}

\appendix

\section{On the interpolation operator.}
We will deduce some results on interpolation operators which satisfy
rather general assumptions. They are satisfied, e.g., by the
Scott–Zhang operator. We work now in a general $d$-dimensional
setting, i.e., we assume  that $\Omega \subset \setR^d$
is a polyhedral domain with Lipschitz continuous boundary.  Let
$\mathcal{T}_h$ denote a family of shape-regular triangulations,
consisting of $d$-dimensional simplices $K$.
% We denote by $h_K$ the diameter of $K$ and by $\rho_K$ the supremum of
% the diameters of inscribed balls.
We assume that $\mathcal{T}_{h}$ is non-degenerate, i.e.,
$\max _{K \in \mathcal{T}_{h}} \frac {h_K}{\rho_K}\le \gamma_0$.  The
global mesh-size $h$ is defined by $h:=\max _{K \in \mathcal{T}_h}h_K$.
% Let $S_K$ denote the neighborhood of~$K$, i.e., $S_K$ is the union of
% all simplices of~$\mathcal{T}$ touching~$K$.
By the assumptions we
obtain that $|S_K|\sim |K|$ and that the number of patches $S_K$ to
which a simplex belongs are both bounded uniformly in $h$ and $K$.
The finite element space $X_h$ is given by
\begin{align*}
  X_h:=\set{\bv \in L^1_{\loc}(\Omega) \fdg \bv \in \mathcal
  P_K(\mathcal{T}_h)}\,,
\end{align*}
where $\mathcal P_{r_0}(\mathcal{T}_h) \subset \mathcal P_K(\mathcal{T}_h) \subset \mathcal P_{r_1}(\mathcal{T}_h)$
for some $r_0\le r_1 \in \setN_0$.

We assume that the interpolation operator $P_h$ is $W^{\ell,1}$-stable. 

\begin{assumption}
\label{ass:interpolation}
Let $\ell_0 \in \setN_0$ and let $P_h \colon (W^{\ell_0,1}(\Omega))^d \to (X_h)^d$.

\noindent (a) For some $\ell\geq \ell_0$ and $m\in\setN_0$ holds uniformly in $K\in T_h$ and
$\bv\in (W^{l,1}(\Omega))^d$
\begin{equation*}
  \sum_{j=0}^mh_K^j \dashint\limits_K| \nabla^j P_h \bv |\, d \bx \leq c ( m , \ell )\sum_{k=0}^\ell
  h_K^k\dashint\limits_{S_K} | \nabla^k \bv |\, d \bx,
\end{equation*}
\noindent (b) For
all $\bv \in  (\mathcal{P}_{r_0} )^d (\Omega)$ holds 
\begin{equation*}
P_h\bv = \bv.
\end{equation*}
\end{assumption}
Note that we have to choose $\ell_0\ge 1$, if the operator $P_h$ is
preserving the boundary values, i.e., $P_h : (W^{\ell_0,1}(\Omega))^d \to
(X_h\cap W^{1,1}_0(\Omega))^d $. Otherwise we allow $\ell_0=0$. 

% \marginpar{maybe this is not needed at this level}
% We note that to shorten notation we write, 
% \begin{equation*}
%    \sum_{j=0}^m\|\nabla^j \bv\|_{L^1}:=   \sum_{j=0}^m\sum_{|\alpha|=j}\|D^\alpha \bv\|_{L^1},
% \end{equation*}
% for a multi-index $\alpha=(\alpha_1,\dots,\alpha_d)$ and $D^\alpha
% \bv=\frac{\partial^{|\alpha|}\bv}{\partial^{\alpha_1}_{x_1}\dots\partial^{\alpha_d}_{x_d}}$. 

% \comment{we need Cor.~\ref{cor:Ocont} and
%   Lem.~\ref{lem:interpolation}, otherwise you are free: L. I do not understand this}

%%%%%%%%%%%%%%%

The properties of the interpolation
operator $P_h$ are discussed in detail in~\cite[Sec.~4,5]{dr-interpol},~\cite[Sec.~3.2]{bdr-phi-stokes}. Let us now prove the two additional
features formulated in Proposition~\ref{prop:Ph} (ii), (iii). We start
with the following non-homogeneous approximation property of $P_h$ (Note that
Proposition~\ref{prop:Ph} (ii) is a special case of the result below).

  \begin{proposition}\label{prop:int-error}
    Let $P_h$ satisfy Assumption~\ref{ass:interpolation} with
    $\ell \le r_0+1$ and let \linebreak ${r,q\in[1,\infty)}$ be such that
    $W^{\ell,q}(\Omega) \vnor \vnor W^{m,r}(\Omega)$. Moreover, assume
    that $h\sim h_K$ uniformly in $\mathcal T_h$. Then, there
    exists a constant ${c=c(\ell,m,q,r,r_0,r_1,\gamma_0)}$ such that
  \begin{align}\label{eq:int-error1}
    \sum_{j=0}^mh^j\,\|\nabla^j(\bv-P_h\bv)\|_r\le c\,  \sum_{k=0}^\ell
    h^{\ell +d\min\{0, \frac 1r -\frac 1q\}}\,\|\nabla^k \bv\|_q\,.
  \end{align}
\end{proposition}

To prove Proposition~\ref{prop:int-error} we start by deriving from
Assumption~\ref{ass:interpolation} the non-homogeneous Sobolev stability adapting the
approach in the case of Orlicz stability from~\cite{dr-interpol}
(cf.~\cite[Thm.~3.1]{zhang-scott} for the classical approach).
\begin{lemma}
  \label{lem:lemmaA1}
  Let $P_h$ satisfy Assumption~\ref{ass:interpolation} and let
  $r,q\in[1,\infty)$ be given. Then there exists $c=c(\ell,m,r_0,r_1)$
  such that for all $K\in \mathcal{T}_h$
  \begin{equation*}
    \sum_{j=0}^mh_K^j \left(\dashint\limits_K |\nabla ^j
      P_h\bv|^r\,d\bx\right)^{\frac 1r}\leq c\,\sum_{k=0}^l h_{K}^k
    \left(\dashint\limits_{S_K} 
      |\nabla^k \bv|^q\,d\bx\right)^{ \frac 1q}, 
  \end{equation*}
  or, in a non-averaged way, this can be formulated as follows
\begin{equation*}
  \sum_{j=0}^mh_K^j\,\|\nabla^jP_h \bv\|_{r,K}\leq c\, h_K^{\,d\left(\frac{1}{r}-\frac{1}{q}\right)} \sum_{k=0}^l
  h_{K}^k \|\nabla^k \bv\|_{q,S_{K}}.
\end{equation*}
\end{lemma}
\begin{proof}
We can write, using \eqref{eq:inverse}, \eqref{eq:int-error1}, and
H\"older's inequality
  \begin{equation*}
    \begin{aligned}
      \sum_{j=0}^mh_K^j\left(\dashint\limits_K |\nabla ^j
        P_h\bv|^r\,d\bx\right)^{\frac 1r} &\le \sum_{j=0}^mh_K^j \,\|\nabla ^j
      P_h\bv\|_{\infty,K}\leq c\,\sum_{j=0}^mh_K^j \dashint\limits_{K}
      |\nabla^j P_h\bv|\,d\bx
      \\
      &\leq c\sum_{k=0}^\ell h_{K}^{k}\dashint\limits_{S_{K}}|\nabla^k
      \bv|\,d\bx\leq c\sum_{k=0}^\ell
      h_{K}^{k}\left(\dashint\limits_{S_{K}}|\nabla^k
        \bv|^q\,d\bx\right)^{\frac 1q}\,.
    \end{aligned}
  \end{equation*}
\end{proof}
Next, we prove a generalized \Poincare-Sobolev-Wirtinger  inequality 
\begin{lemma}
  \label{lem:lemmaA2}
  Let $\ell\in\setN$ and $q,r\in[1,\infty)$ be such that
  $W^{\ell,q}(K)\hookrightarrow\hookrightarrow W^{m,r}(K)$. Then, there
  exists a constant $c=c(\ell,m,\gamma_0,r,q)$ such that for all
  $\bv\in W^{\ell,q}(K)$ with
  $\dashint\limits_{K}\nabla^k \bv\,d\bx=\bfzero$ for
  $k=0,\dots,\ell-1$ it holds
  \begin{equation*}
      \sum_{j=0}^mh_K^j \, \|\nabla ^j\bv\|_{r,K}\leq c\, h_K^{\ell+d\left(\frac{1}{r}-\frac{1}{q}\right)}\|\nabla^\ell \bv\|_{q,K}.
  \end{equation*}
\end{lemma}
\begin{proof} 
  Let us first show that for every $j=0,\ldots, m$ there exists $c_j>0$ (depending on $K$) such
  that there holds $\|\nabla ^j \bv\|_{r,K}\leq
  c_j\|\nabla^\ell \bv\|_{q,K}$. Fix $j$ and assume per absurdum
  that there exists $\{\widetilde{\bv}_n\}\subset W^{\ell,q}(K)$ such that
  \begin{equation*}
         \|\nabla ^j\widetilde{\bv}_n\|_{r,K}> n\|\nabla^\ell\widetilde{\bv}_n\|_{q,K}.
  \end{equation*}
Setting  $\bv_n:=\frac{\widetilde{\bv}_n}{\|\nabla
  ^j\widetilde{\bv}_n\|_{r,K}}$  we get 
\begin{equation}
  \label{eq:imbedding}
  \|\nabla ^j\bv_n\|_{r,K}=1\qquad\text{and}\qquad   \|\nabla^\ell \bv_n\|_{q,K}<\frac{1}{n}.
\end{equation}
Note that  $\|\bw\|_{q,K}+\|\nabla^{\ell-j} \bw\|_{q,K}$ is an equivalent norm
on $W^{\ell-j,q}(K)$ (cf.~\cite[p.~179]{trieb}).
We have to distinguish the cases $r\ge q$ and $r<q$.

\noindent\textbf{Case 1:} $r\geq q$. The sequence $\{\nabla ^j\bv_n\}$ is bounded in
$W^{\ell-j,q}(K)\hookrightarrow\hookrightarrow L^{r}(K)$, hence there exists a sub-sequence
(relabelled as $\{\nabla ^j \bv_n\}$) such that $\nabla ^j \bv_n\to \bV$ strongly in $L^r(K)$ and
$\|\bV\|_{r,K}=1$. This and~\eqref{eq:imbedding} imply that
$\{\nabla ^j\bv_n\}$ is a Cauchy
sequence in $W^{\ell-j,q}(K)$, hence  $\nabla ^j\bv_n\to \bW$ in $W^{\ell-j,q}(K)$. Uniqueness of
the limit implies that $\bW=\bV$. This proves that $\nabla ^j \bv_n\to
\bV$ in $W^{\ell-j,q}(K)$. Moreover, \eqref{eq:imbedding} implies
$\|\nabla^{\ell-j} \bV\|_{q,K}=0$, hence that $\bV\in
\mathcal{P}_{\ell-j-1}$. Next, the convergence in 
$W^{\ell-j,q}(K)$ implies that also the averages converge. Hence
\begin{equation*}
  \bfzero =\dashint\limits_K \nabla ^k \nabla ^j\bv_n\,d\bx \to \dashint\limits_K \nabla^k \bV\,d\bx\qquad \text{for
  }k=0,\dots,\ell-j-1, 
\end{equation*}
but as $\bV$ is polynomial of degree less or equal than $\ell-j$, this implies that $\bV=\bfzero$.
Thus,  $\|\bV\|_{r,K}=0$, contradicting the fact that $\|\bV\|_{r,K}=1$.

\noindent\textbf{Case 2:} $r< q$.  In this case the same argument as in the previous case shows that 
\begin{equation*}
  \|\nabla ^j\bv\|_{q,K}\leq c_j\|\nabla^\ell \bv\|_{q,K},
\end{equation*}
and then by H\"older's inequality 
\begin{equation*}
  \|\nabla ^j\bv\|_{r,K}\leq c \,\|\nabla ^j\bv\|_{q,K}\leq c\,c_j\,\|\nabla^l \bv\|_{q,K},
\end{equation*}
with $c=c(K)$.

To prove how the constants $c_j$ depend on $K$ we proceed as follows: 
We pass from a generic simplex $K$ to the reference simplex $\widehat{K}$,
use the previous inequalities in the reference domain with constants depending only
$\widehat{K}$, and then we come back to the original simplex $K$. This
shows for every $j=0,\ldots , m$
\begin{equation*}
  \begin{aligned}
    \|\nabla ^j \bv\|_{r,K}^r&\simeq h_K^d\|\widehat{\nabla ^j\bv}\|_{r,\widehat{K}}^r
    \\
    &    \leq h_K^d
    c_j(\widehat{K})\|\widehat{\nabla}^{\ell-j}\widehat{\nabla ^j\bv}\|_{q,\widehat{K}}^r
    \simeq h_K^d
    c_j({\widehat{K}})h_{K}^{\left(\ell-j-\frac{d}{q}\right)r}\|{\nabla}^\ell
    {\bv}\|_{q,{K}}^r.
  \end{aligned}
\end{equation*}
Hence, we get 
\begin{equation*}
    h_K^j\|\nabla ^j\bv\|_{r,K} \leq c_j({\widehat{K}}) h_K^{\ell+d\left(\frac{1}{r}-\frac{1}{q}\right)}\|\nabla^\ell\bv\|_{q,K}\,,
  \end{equation*}
  which implies the assertion with $c=\sum_{j=0}^mc_j(\widehat K)$.
\end{proof}
We recall now a \Poincare-Wirtinger type inequality, where it is possible to replace the
average over the whole domain $G$ with that one over a sub-domain $A\subset
G$ (cf.~\cite[Cor.~8.2.6]{lpx-book}, \cite[Ch.~7.8]{gilbarg-trudinger}), provided that $G$ is an
$\alpha$-John domain and 
\begin{equation*} 
  |A| \simeq |G|\,.
\end{equation*}
Note that, due to our assumptions on the triangulation, we have that
$S_K$ are $\alpha$-John domains, where $\alpha$ depends only on
$\gamma_0$, and that $\abs{K }\simeq \abs{S_K}$ for all $K \in
\mathcal{T}_h$.
\begin{lemma}
  \label{lem:lemmaA3}  There exists a constant $c=c(d,\gamma_0)$ such that 
  \begin{equation*}
    \left\|\bv-\mean{\bv}_{{K}}\right\|_{q,S_{K}}\leq c\, h_K\,\|\nabla \bv\|_{q,S_{K}}\qquad
    \forall\, \bv\in W^{1,q}(S_{K})\,.
  \end{equation*}
\end{lemma}
% The above result is well known (cf. also~), but here is relevant
% the dependence on the measure of $K$, which is comparable to that of $S_{K}$ for uniform
% and non-degenerate meshes. In particular,  if $\dashint\limits_K \bv=0$, then
% $\|\bv\|_{q,S_{K}}\leq c\, h\|\nabla \bv\|_{q,S_{K}}$.

% We now prove a result of local Sobolev approximation, in which the summability exponent of
% the two sides of the inequality is not equal.
This enables us to prove a local variant of
Proposition~\ref{prop:int-error}. 
\begin{lemma}
  \label{lem:lemmaA4}
    Let $P_h$ satisfy Assumption~\ref{ass:interpolation} with
    $\ell \le r_0+1$ and let $r,q\in[1,\infty)$ be such that
    $W^{\ell,q}(\Omega) \vnor \vnor W^{m,r}(\Omega)$. Then there exists
  $c=c(\ell,m,r_0,r_1,\gamma_0,r,q,d)$ such that for all $\bv \in
  W^{\ell,q}(\Omega)$ and all $K \in
  \mathcal{T}_h$
  \begin{equation*}
     \sum_{j=0}^mh_K^j \, \|\nabla ^j\bv- \nabla ^jP_h \bv\|_{r,K}\leq c\, h_K^{\ell+d\left(\frac{1}{r}-\frac{1}{q}\right)}\|\nabla^\ell \bv\|_{q,S_K}.
  \end{equation*}
\end{lemma}
\begin{proof}
  We split the interpolation error by adding and subtracting a
  polynomial $\boldsymbol p$ of degree
  less than $\ell$ and use Assumption~\ref{ass:interpolation} (b) and
  Lemma~\ref{lem:lemmaA1} to get for all $j=0,\ldots, m$
  \begin{equation*}
    \begin{aligned}
      &\sum_{j=0}^mh_K^j \, \|\nabla ^j\bv- \nabla ^jP_h \bv\|_{r,K}
      \\
      &\leq \sum_{j=0}^mh_K^j \, \|\nabla ^j\bv- \nabla ^j\boldsymbol
      p\|_{r,K}  +\sum_{j=0}^mh^j \, \|\nabla ^j P_h\big ( \bv
      -\boldsymbol p\big)\|_{r,K}
      \\
      &\leq \sum_{j=0}^mh_K^j \, \|\nabla ^j\bv- \nabla ^j\boldsymbol p\|_{r,K}+c\, h_K^{d\left(\frac{1}{r}-\frac{1}{q}
        \right)}\sum_{k=0}^\ell h_K^k\|\nabla^k(\bv-\boldsymbol p)\|_{r,S_K}\,.
    \end{aligned}
  \end{equation*}
Since $l\leq r_0+1$ we can use Lemma~\ref{lem:lemmaA2} to infer that for all polynomials $\boldsymbol p$ such that $\dashint\limits_K\nabla^k
\bv\, d\bx=\dashint\limits_K\nabla^k\boldsymbol p\,d\bx$, for
$k=0,\dots,\ell-1$, we have 
\begin{equation*}
  \sum_{j=0}^mh_K^j \, \|\nabla ^j\bv-\nabla ^j \boldsymbol p\|_{r,K}\leq c\,
  h_K^{\ell+d\left(\frac{1}{r}-\frac{1}{q}\right)}\|\nabla^\ell
  \bv\|_{q,K} \le c\,  h_K^{\ell+d\left(\frac{1}{r}-\frac{1}{q}\right)}\|\nabla^\ell
  \bv\|_{q,S_K} \,.
\end{equation*}
For the same polynomials we have $\dashint _K \nabla ^k (\bv
-\boldsymbol p)\, d\bx =0$, $k=0,\dots,\ell-1$, and thus,
Lemma~\ref{lem:lemmaA3} yields for $k=0,\dots,\ell-1$
$$
  \|\nabla^k(\bv-\boldsymbol p)\|_{q,S_K}\leq c
  \,h_K^{\ell-k}\|\nabla^\ell(\bv-\boldsymbol p)\|_{q,S_K}=c
\,h_K^{\ell-k}\|\nabla^\ell\bv\|_{q,S_K}\,.
$$
The last three inequalities prove the assertion. 
%  In this way we can estimate $\|\bv-P_h \bv\|_{r,K}$ by terms
% involving just the derivatives of order exactly equal to $l$
% \begin{equation*}
%       \|\bv-P_h \bv\|_{r,K}\leq   c\, h_K^{d\left(\frac{1}{r}-\frac{1}{q}
% \right)}h_K^l\|\nabla^l\bv\|_{q,S_K}, 
% \end{equation*}
% ending the proof.
\end{proof}
We now have all results to prove Proposition~\ref{prop:int-error}. 
% \begin{lemma}
%   \label{lem:lemmaA5}
%   Let $l\leq r_0+1$ and let $W^{l,q}(\Omega)\hookrightarrow\hookrightarrow L^r(\Omega)$,
%   then there exists $c=c(l,q,r,\Omega)$ such that 
%   \begin{equation*}
%     \|\bv-P_h\bv\|_{r,\Omega}\leq c \,
%     h^{l+\min\left\{0,\frac{1}{r}-\frac{1}{q}\right\}}\|\nabla^l \bv\|_{q,\Omega}.
%   \end{equation*}
% \end{lemma}
\begin{proof}[Proof of Proposition~\ref{prop:int-error}]
  We split the integration over $\Omega$ into a sum over $K$, and then
  we use Lemma~\ref{lem:lemmaA4} to get for each $j \in \set{0,\ldots,
  m}$
 \begin{align*}
   % \sum_{j=0}^m
   \|\nabla ^j\bv- \nabla ^jP_h  \bv\|_{r,\Omega}^r
   &=%\sum_{j=0}^m
     \sum_{K\in \mathcal{T}_h}   \|\nabla ^j\bv- \nabla ^jP_h \bv\|_{r,K} ^r
   \\
   &\leq c\sum_{K\in \mathcal{T}_h} h_K^{\ell r-jr+dr\left(\frac{1}{r}-\frac{1}{q}\right)}\|\nabla^\ell \bv\|_{q,S_K}^r.
 \end{align*}
 We set now $\alpha_K:=\|\nabla^\ell \bv\|_{q,S_K}^q$ and observe that $\nabla^\ell \bv\in
 L^{q}(\Omega)$ is equivalent to $\alpha_K\in \ell^1(\setN)=\ell^1$. We use H\"older inequality
 in the $\ell^q$ spaces to estimate the right-hand side.
We distinguish again the two cases $q\le r$ and $q>r$.

\noindent\textbf{Case 1:} $q\leq r$. 
In this case, since $\frac{r}{q}-1\geq0$ and since for $\{a_n\}\subset \ell^1$ it holds
$\|a_n\|_{\ell^\infty}\leq \|a_n\|_{\ell^1}$, we can write
\begin{equation*}
  \begin{aligned}
    \sum_{K\in\mathcal{T}_h}\|\nabla^\ell \bv\|_{q,S_{K}}^{r}&=
    \sum_{K\in\mathcal{T}_h}\alpha_K^{\frac{r}{q}}=
    \sum_{K\in\mathcal{T}_h}\alpha_K\,\alpha_K^{\frac{r}{q}-1} 
    \\
    &\leq
    \sum_{K\in\mathcal{T}_h}\alpha_K\,\|\alpha_K\|_{\ell^\infty}^{\frac{r}{q}-1}\le 
    \|\alpha_K\|_{\ell^\infty}^{\frac{r}{q}-1} \|\alpha_K\|_{\ell^1}    \leq     \|\alpha_K\|_{\ell^1}^{\frac{r}{q}}.
  \end{aligned}
\end{equation*}

\noindent\textbf{Case 2:} $q> r$.
In this case 
\begin{equation*}
    \sum_{K\in\mathcal{T}_h} \alpha_K^{\frac{r}{q}}\leq\left(\sum_{K\in\mathcal{T}_h}\alpha_K\right)^{\frac{r}{q}}
    \left(\sum_{K\in\mathcal{T}_h}1\right)^{1-\frac{r}{q}}=\left(\sum_{K\in\mathcal{T}_h}\alpha_K\right)^{\frac{r}{q}}
    \left(\#K\right)^{1-\frac{r}{q}} \,.
\end{equation*}
Since $h \sim h_K$ uniformly in $\mathcal T_h$ we get $|\Omega|\simeq
\#K\, h^d$ and thus $\#K\sim h^{-d}$. Hence we obtain 
\begin{equation*}
    \sum_{K\in\mathcal{T}_h} \alpha_K^{\frac{r}{q}}\leq  c\,h^{-dr\left(\frac{1}{r}-\frac{1}{q}\right)}
    \|\alpha_K\|_{\ell^1}^{\frac{r}{q}}.
\end{equation*}
Putting the two cases together, using $h_K \le h$ and
$W^{\ell,q}(\Omega)\vnor W^{m,r}(\Omega)$, we proved for each $j \in
\set{0,\ldots, m}$
 \begin{equation*}
   % \sum_{j=0}^m
   \|\nabla ^j\bv- \nabla ^jP_h  \bv\|_{r,\Omega}^r
   \leq c\, h^{\ell r-jr+dr\max\left\{0,\left(\frac{1}{q}-\frac{1}{r}\right)\right\}}
   \left(\sum_{K\in \mathcal{T}_h}\|\nabla^\ell \bv\|_{q,S_K}^q\right)^{\frac{r}{q}}\,.
%\leq  c\,  h^{lr+dr\min\left\{0,\left(\frac{1}{r}-\frac{1}{q}\right)\right\}}\|\nabla^l \bv\|_{q,\Omega}^r. 
 \end{equation*}
Taking the $r$-th root, multiplying by $h^{j}$ and summing up over $j=0,\ldots,m$ proves the assertion, since $\left(\sum_{K\in
    \mathcal{T}_h}\|\nabla^\ell
  \bv\|_{q,S_K}^q\right)^{\frac{1}{q}}\le c\,
\norm{\nabla ^\ell \bv}_{q,\Omega}$.
\end{proof}

Finally, we prove the following version of the continuity of the interpolation
operator $P_h$ in Orlicz spaces.
%\marginpar{p10-11-12(after k-summation)}

\begin{lemma}\label{lem:interpolation}
Let $\bF(\bD \bv) , \bF(\bD \bw) \in
    W^{1,2}(\Omega)$. Then there exists a constant
    $c=c(p,r_1,\gamma_0)$ such that % for almost all $s,t \in I_m$, $m=1,\ldots, M$, we have
% \comment{In fact here we are not using any time regularity, this comes later on with
%   integration. Should be better say $\bF(\bD \bv(t))\in  W^{1,2}(\Omega)$ for a.e. $t\in
%   I$??

% This has to be done accordingly with statement of Lemma~\ref{lem:interpolation}.
%}
\begin{align*}%  \label{eq:interpolation}
  \begin{aligned}
    &\int\limits_\Omega \phi_{|\bD\bfv|}\big(\big|\bD P_h{\bv}
    -\bD P_h\bfw \big|\big)\,\dx
\leq c \,h^2\|\nabla\bF(\bD\bv)\|^2_2
    +c\,\|\bF(\bD\bv)-\bF(\bD\bw)\|^2_2 \,,
      \end{aligned}
\end{align*}
where the constants depends only on $\gamma_0$ and $p$.
\end{lemma}
  \begin{proof}
    %It is at this point that we need to use the special properties of the finite element
    %approximation.\comment{which ones?}
    Using again  $\int_\Omega f \,d\bx =\sum_{K \in \mathcal{T}_h}
    \int_K f \, d\bx$, it suffices to treat one simplex $K$.
    We obtain, thanks to back and forth shift changes
    (cf.~Proposition~\ref{lem:hammer} (ii)), the properties of the
    interpolation operator, Korn's inequality
    (cf.~\cite[Thm.~6.13]{john}) and
    Proposition~\ref{lem:hammer} (i), Poincar\'e's inequality applied
    to $\bF(\bD\bu)$ in $L^2(S_K)$ and Proposition~\ref{lem:hammer} (iii),  the
    properties of the triangulation, the following chain of inequalities
    \allowdisplaybreaks
%     \begin{align*}
% %      \begin{aligned}
%         &\int\limits\limits\limits_\Omega
%         \phi_{|\bD\bfu(s)|}\big(\big|\bD(P_h[{\bu}(s)-\bfu(t_{m})])\big|\big)\,dx
%         =
%         \sum_{K}\int\limits\limits\limits_K\phi_{|\bD\bfu(s)|}\big(\big|\bD(P_h[{\bu}(s)-\bfu(t_{m})])\big|\big)\,dx
%     \end{align*}
%     proceed locally
    \begin{align*}
      &\int\limits\limits\limits_K
       \phi_{|\bD\bfv|}\big(\big|\bD P_h\bv -\bD P_h\bfw\big|\big)\,d\bx
        \\
        &\leq
        c_\delta\int\limits\limits_K\phi_{|\langle\bD\bfv\rangle_{S_K}|}\big(\big|\bD P_h\bv -\bD P_h\bfw\big|\big)\,d\bx
        +\delta \int\limits\limits_K\phi_{|\langle\bD\bfv\rangle_{S_K}|}\big(\big|\bD\bv-\langle\bD\bv\rangle_{S_K}\big|\big)\,d\bx
        \\
        &\leq
        c\, c_\delta\int\limits\limits_{S_K}\phi_{|\langle\bD\bfv\rangle_{S_K}|}\big(\big|\nabla\bv-\nabla\bfw\big|\big)\,d\bx
        +\delta\int\limits\limits_K\phi_{|\langle\bD\bfv\rangle_{S_K}|}\big(\big|\bD\bv-\langle\bD\bv\rangle_{S_K}\big|\big)\,d\bx
        \\
        &\leq
        c\, c_\delta\int\limits\limits_{S_K}\phi_{|\langle\bD\bfv\rangle_{S_K}|}\big(\big|\bD\bv-\bD\bfw)\big|\big)\,d\bx
        +\delta\int\limits\limits_K\phi_{|\langle\bD\bfv\rangle_{S_K}|}\big(\big|\bD\bv-\langle\bD\bv\rangle_{S_K}\big|\big)\,d\bx
        \\
        &\leq
        c_\delta\int\limits\limits_{S_K}\phi_{|\bD\bfv|}\big(\big|\bD\bv-\bD\bfw)\big|\big)\,d\bx
        +\delta\int\limits\limits_{S_K}\phi_{|\langle\bD\bfv\rangle_{S_K}|}\big(\big|\bD\bv-
        \langle\bD\bv\rangle_{S_K}\big|\big)\,d\bx
        \\
        &\leq c_\delta\,\|\bF(\bD\bw)-\bF(\bD\bv)\|^2_{2,S_K}
        +\delta\int\limits\limits_{S_K}|\bF(\bD\bv)-\langle\bF(\bD\bv)\rangle_{S_K}|^2\, d\bx
        \\
        &\leq c_\delta\,\|\bF(\bD\bw)-\bF(\bD\bv)\|^2_{2,S_K} +c\,\delta\, h^2
        \int\limits\limits_{S_K}|\nabla\bF(\bD\bv)|^2\, d\bx.
%      \end{aligned}
    \end{align*}
    This yields the assertion.
  \end{proof}
\section*{Acknowledgments}
The research of Luigi C. Berselli that led to the present paper was partially supported
by a grant of the group GNAMPA of INdAM and by the project of the University of Pisa
within the grant PRA$\_{}2018\_{}52$~UNIPI \textit{Energy and regularity: New techniques
  for classical PDE problems.} 
%\bibliography{fluid,rose}%lista,libri,altri}
%\bibliographystyle{plain}

\def\cprime{$'$} \def\cprime{$'$} \def\cprime{$'$}

\end{document}